\documentclass[preprint,12pt]{article}

\usepackage{amssymb}

\usepackage{color}

\usepackage{amsmath}
\usepackage{amsthm}
\usepackage{amsfonts, dsfont}
\usepackage[dvips]{graphicx}

\usepackage[english]{babel}

\usepackage{lineno}

\newcommand{\red}{\color{black}}
\newcommand{\rred}{\color{black}}

\newcommand{\RM}[1]{\MakeUppercase{\romannumeral #1{.}}}

\newcommand {\oo}{\mbox{\boldmath $0$}}
\newcommand {\pprime}{{2}{}}
\newcommand {\uu}{\mbox{\boldmath $u$}}

\newcommand {\vv}{\mbox{\boldmath $v$}}
\newcommand {\ww}{\mbox{\boldmath $w$}}

\newcommand {\ee}{\mbox{\boldmath $e$}}

\newcommand {\nn}{\mbox{\boldmath $n$}}

\newcommand {\VV}{\mbox{\boldmath $V$}}

\newcommand {\RR}{{\Bbb R}}
\newcommand {\ZZ}{{\cal Z}}
\newcommand {\II}{{\Bbb I}}
\newcommand {\EE}{{\Bbb E}}
\newcommand {\JJ}{{\Bbb J}}
\newcommand {\MM}{{\Bbb M}}
\newcommand {\eij}{\hat{e}}
\newcommand {\pp }{\mbox{\boldmath $\psi $}}
\newcommand {\hfrac}{\mbox{$\frac{h^1}{h^2}$}}
\newcommand {\hfraci}{\mbox{$\frac{h^2}{h^1}$}}\newcommand {\hfracdi}{\mbox{$\frac{h^2-h^1}{h^1}$}}
\newcommand {\barh}{\mbox{$\bar{h}$}}
\newcommand {\dive}{\mbox{div}}

\newcommand {\z }{\mbox{\boldmath $z$}}

\newcommand {\veps}{{\varepsilon}}
\newcommand{\fr}{)\hspace{-0.6mm})}
\newcommand{\fl}{(\hspace{-0.6mm}(}
\newcommand {\jedna}{\mathds{1}}

\def\softd{{\leavevmode\setbox1=\hbox{d}%
          \hbox to 1.05\wd1{d\kern-0.4ex{\char039}\hss}}}

\newtheorem{theorem}{Theorem}[section]

\theoremstyle{plain}
\newtheorem*{theorem*}{Theorem}

\newtheorem{lemma}{Lemma}[section]
\newtheorem{corollary}{Corollary}[section]
\newcounter{remark}[section]
\renewcommand{\theremark}{\thesection.\arabic{remark}}
\newenvironment{remark}{\refstepcounter{remark}{\bf
Remark~\theremark.}}{\hfill$\diamond$}

\begin{document}




\title{On the convergence of  fixed point iterations\\ for the moving geometry \\in a fluid-structure interaction problem}


\author{Anna Hundertmark\\ Institute of Mathematics, University of Mainz, Germany}


\maketitle
\begin{abstract}
In this paper  a fluid-structure interaction problem for the incompressible Newtonian fluid is studied. We prove the convergence of an iterative process with respect to the computational domain geometry. 
    In our previous  works on numerical approximation of similar problems we refer this  approach as  the global iterative method \cite{LZ2008, LZ2010}.
    This iterative approach can be understood as a linearization  of  the so-called geometric nonlinearity of the underlying model.  The proof of the convergence is based on  the Banach fixed point argument, where the contractivity of the corresponding mapping is shown due to the continuous dependence of the weak solution on the given domain deformation.
   This estimate is obtained by  remapping the  problem onto a fixed domain and using appropriate divergence-free test functions involving the difference of two solutions.
\end{abstract}

{\bf keywords}: {fluid-structure interaction,  fixed point method, continuous dependence on data, uniqueness, Banach fixed point theorem, hemodynamics, incompressible Newtonian fluid}





\section{Introduction}

Mathematical analysis and numerical simulation of fluid-structure interaction (FSI) problems  is an intensively studied part of the computational fluid dynamics.
In FSI problems the computational domain deforms under the fluid forces. Thus, the domain deformation, governed by a structure equation, depends on the solution of the fluid equation.  In the literature this dependence  is referred to as {\em the geometric nonlinearity}. One possible  strategy  to  find a solution to such coupled problem is a linearization of the problem with respect to the geometric nonlinearity. The problem is solved at first for a known domain, deformed according to a given,  sufficiently smooth deformation function.  The second step is to prove  the existence of a fixed point for the mapping between the domain deformation and the solution.

In our previous joined works on this topic  we proved the existence and uniqueness of the weak solution for an  approximation of the fully coupled fluid-structure interaction problem for the  Navier-Stokes equations and a parabolic (viscoelastic) equation for the boundary deformation, see \cite{FZ10, ZAUdiss}. In these works a construction of  the unknown domain  through the  iterative process with respect to the domain geometry has been proposed. 
For a  semi-pervious approximation of the coupling condition,
the convergence of  this iterative process has been shown in \cite{FZ10} using the Banach fixed point argument. However, it was not possible to prove the convergence of this iterative process for the fully coupled original problem due to the lack of regularity of the domain deformation. 

Further results on the existence of a weak solution to the  fully coupled FSI problem for the Newtonian fluid and a  viscoelastic/elastic plate in 2D and 3D have been obtained by Grandmont et al. \cite{GRA05,GRA08}. In the case of a two-dimensional fluid flow and one-dimensional structure the authors obtained in \cite{GRA05,GRA08} the existence of a weak solution for the same regularity of the domain deformation  as in  \cite{FZ10} using the Schauder fixed point argument, without giving any details on the construction of the domain deformation 
and without providing  the uniqueness of the solution.
 Using a similar fixed point procedure for the geometry the existence of weak solution for a fluid-structure interaction problem for a generalized power-law shear-dependent fluid including the Newtonian case has been shown in our recent joint works \cite{HLN14, HLN14book}, see also \cite{ LENGpreprint}.

The continuous dependence of
the weak solution on the initial data and the uniqueness have been shown for the incompressible Newtonian fluid by Padula et al. in \cite{PAD10}. They considered a transformation of the solution between two  domains with different deformations that  preserves the solenoidal property of the solution.
For  further works on the mathematical analysis of similar FSI problems for Newtonian as well as non-Newtonian fluids  we refer also to \cite{LENGpreprint,
 PAD08,  COUTSHKO,CHENSHKO, NEU, LEQ, LENRU,CANICMM, CANICMMM, MUHAC, CANICM}
  and the references therein. 
\medskip
   
   In this work  
   we will deal with the geometric nonlinearity and study the iterative process with respect to the domain geometry for the fully coupled FSI problem presented in the following section.

\subsection{Mathematical model}
  
We consider a fluid flow problem in a two-dimensional channel representing a longitudinal cut of a vessel.
The fluid flow is governed by the momentum and the continuity
equation, written in  Cartesian coordinates  for the sake of simplicity,
\begin{eqnarray}
 {\rho\partial_t \vv}
+ \rho \left(\vv\cdot \nabla\right)\vv
-\dive \big[2{\mu}\, e( \vv)\big]
+\nabla p
&=&0,\label{i.1}
 \\\mbox{div}\,\vv&=&0.\nonumber
\end{eqnarray}
Here $\rho$ denotes the constant density of the fluid, $\vv=(v_1,v_2)^T$ the velocity vector, 
$p$ the pressure and $e(\vv)=\frac{1}{2}(\nabla \vv+\nabla \vv^T)$ the symmetric deformation tensor. We consider the fluid viscosity $\mu$   to be  constant.

The radial vessel wall deformation $\eta$ is modelled using the one-dimensional  viscoelastic string model. The  deformation $\eta$ is a function of longitudinal variable $x_1$ and time $t$  governed by the following structure equation
\begin{eqnarray}\label{i.55}
&&{\cal E}\rho\left[
\frac{\partial^2\eta }{\partial t^2}
-a\frac{\partial^2\eta }{\partial x_1^2}+b\eta
 + c\frac{\partial^5\eta }{\partial t\partial x_1^4}
-a\frac{\partial ^2 R_0}{\partial x_1^2}\right](x_1,t)
= \\
&&\hphantom{\hspace{20mm}}
\Big[-{\bf T}_f\nn\cdot\ee_2-P_w\Big](x_1,R_0(x_1)+\eta(x_1,t),t),\nonumber
\end{eqnarray}
 $x_1 \in(0,L), t \in (0,T)$.
Here ${\cal E}=\rho_w\hbar 
 \sqrt{1+(\partial_{x_1}R_0)^2}$, where $\rho_w$ is the density of the wall tissue and $\hbar$ its thickness. 
We assume that $\cal E$ is bounded and $a, \, b, \, c$  are positive constants describing the mechanical properties of the vessel wall tissue. Further, $R_0(x_1)$ stands for the reference radius of the cylinder (vessel), $P_w$  the external pressure acting on the deformable vessel wall,  $\nn$ 
 the normal vector to the moving  vessel wall and ${\bf T}_f=2\mu e(\vv)-p{\bf I}$  the Cauchy stress tensor. Similar  models for the wall deformation have been considered in  {\cite{GRA05,CANICM}, see also \cite{QUATUVE,QUA,QUAcla}}.
 \begin{figure}
\begin{center}
  \label{fig:domain}
  \includegraphics[scale=0.364]{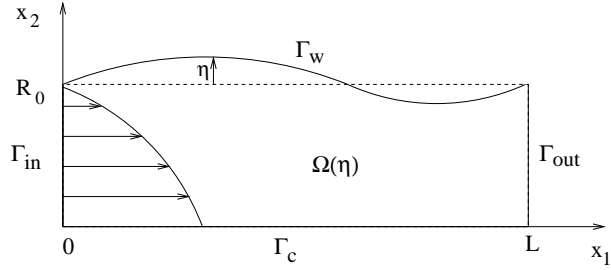}
\caption{Deformable computational domain}
\end{center}
 \end{figure}

The moving wall deforms under the fluid action, thus the computational domain for the fluid is determined by the unknown deformation $\eta$, i.e., 
 \[
\Omega(\eta(t))\equiv \left\{ (x_1,x_2); \ 0<x_1<L , \  0<x_2<R_0(x_1)+\eta(x_1,t)\right\}, \ 0<t<T,
\] where $R_0(x_1)$ 
describes the reference position  of the deformable wall. The boundary of $\Omega(\eta(t))$ consists of one  moving part $\Gamma_w(t)$, \[ \Gamma_w(t)\equiv\{(x_1, x_2);\ x_2=R_0(x_1)+\eta(x_1,t),\, x_1 \in (0,L)\},\]  and three fixed  boundaries $\Gamma_{in}, \, \Gamma_{out} $ and  $\Gamma_c$ 
denoting  the inflow, the outflow  and the bottom boundary, respectively, 
see Fig. \ref{fig:domain}.

The fluid problem (\ref{i.1}) and the structure equation (\ref{i.55}) are coupled through the natural kinematic condition describing the continuity of the velocities on the moving interface $\Gamma_w(t)$,
\begin{eqnarray}\label{dirichlet}
\vv\big(x_1, R_0(x_1)+\eta(x_1,t), t\big)=\big(0,{\partial_t \eta(x_1,t)}\big).
\end{eqnarray}
Moreover, the external force causing the wall deformation arises due to the fluid stress.  Thus, the second coupling condition, expressed by the right-hand side of (\ref{i.55}), has dynamical character  and  can be understood as the continuity of stresses.

We complete the system (\ref{i.1})--(\ref{dirichlet}) with Neumann-type boundary conditions with kinematic pressure at the inflow $\Gamma_{in}$ and outflow boundary $\Gamma_{out}$,
\begin{eqnarray}
\left(2\mu
\frac{\partial v_1}{\partial x_1}
-p+P_{ in} -\frac{\rho}{2}\left|v_1 \right|^2  \right)
(0,x_2,t)=0,\quad v_2(0,x_2,t)=0,\label{i.7}\\
\left(2\mu
\frac{\partial v_1}{\partial x_1}
-p+P_{ out} -\frac{\rho}{2}\left|v_1 \right|^2  \right)
(L,x_2,t)=0,\quad v_2(L,x_2,t)=0.
\label{i.8}
\end{eqnarray}
Here $P_{in},\, P_{out}$ are some given pressures  and  $0<x_2<{R_0(\tilde{x_1})+\eta(\tilde{x_1},t)} $ with $\tilde{x_1}=0,L$.
 On the bottom boundary $\Gamma_c$ the  flow symmetry is considered, 
\begin{equation}
v_2(x_1,0,t)=0 \ , \hspace{5mm} \mu \frac{\partial v_1}{\partial x_2}(x_1,0,t)=0,\quad\  0<x_1<L,\ 0<t<T.
\label{i.9}
\end{equation}
We equip equation (\ref{i.55})  with the following clamped boundary conditions,
\begin{eqnarray}\label{i.6}
&&\eta (0,t)=\eta (L,t)=0,    \hspace{3mm} \hspace{3mm}\eta_{x_1} (0,t)=\eta_{x_1} (L,t)=0.
\end{eqnarray} 
The initial conditions for the fluid (\ref{i.1}) and structure (\ref{i.55}) equations read as follows,
\begin{eqnarray}\label{i.10}
\vv (x_1,x_2,0)&=&\oo \hspace{5mm} \mbox{for any} \ 0<x_1<L, \\&&\hspace{22mm} 0<x_2<{R_0(x_1)},\nonumber
\\\eta (x_1,0)={\red \eta_{x_1}(x_1,0)=}\eta_t(x_1,0)&=&0\hspace{5mm} \mbox{for any} \ 0<x_1<L.\nonumber
\end{eqnarray}
The problem (\ref{i.1})--(\ref{i.10}) represents a mathematical model for the flow of an incompressible Newtonian 
fluid in a deformable domain, where the deformation of the upper wall obeys the viscoelastic string model. Such a  fully coupled fluid-structure interaction problem can be used in  hemodynamics  to model the blood flow in (large) elastic vessels.
The geometry of computational domain $\Omega(\eta)$ depends on unknown solution  $\eta$.


{ As already mentioned,  one possible and commonly used way to show the existence result for such fully coupled FSI problems
  is to linearize the problem with respect to the geometric nonlinearity and consequently to find a fixed point of this nonlinearity. At the beginning, a smooth enough  domain deformation function $\delta$ is given. The computational domain is then approximated by this function,  $ \Omega(\eta) \approx \Omega(\delta)$. Next, existence of a weak solution $(\vv,\eta)$  defined on $\Omega(\delta)$ is proven. Finally, a fixed point of the mapping
  ${\cal F}(\delta)=\eta$ defined by the weak formulation on $\Omega(\delta)$ is found.
  By applying the Schauder fixed point theorem for the mapping ${\cal F}$ it can be shown that at least one
  $\eta^\star$ s.t. $\eta^\star={\cal F}(\eta^\star)$ exists.
  Furthermore,  if the assumptions of  the Banach fixed point theorem are satisfied, the uniqueness of the fixed point and the convergence of the  iterative process ${\cal F}(\eta^{k-1})=\eta^{k}$ can be obtained additionally.
  The Banach fixed point theorem is based on the contractivity of the mapping $\cal{F}$, which in our case follows from the continuous dependence of the weak solution $(\vv, \eta)$ on the given domain deformation function $\delta$.

  Let us note that in \cite{FZ10} the convergence of the iterative process described above and 
 the  uniqueness of the fixed point have been obtained only for a semi-pervious and pseudo-compressible approximation of the original fluid-structure interaction problem. The right-hand side of structure equation  (\ref{i.55}) has been replaced by $\kappa(v_2-\eta_t)$ for a fixed $\kappa \in { \RR}$. This   implies  more regularity for $\eta$ than in the original case  (\ref{i.55}).
  The continuous dependence on data and consequently  the contractivity of the fixed point mapping has been shown in  \cite{FZ10} only for non-solenoidal solutions and fixed $\kappa$. We would like to point out here, that  by letting $\kappa \to \infty$ the original coupled problem with conditions (\ref{i.55}), (\ref{dirichlet}) can be obtained.  However, existence of the weak solution  for this case has been studied in \cite{FZ10} only for the known domain $\Omega(\delta)$.
 
  A significant difficulty in the proof of the continuous dependence on data 
  for  incompressible fluids lies in the requirement 
  of divergence-free test functions involving the difference of two solutions defined in two different domains. 
One has to test with transformed solutions, as in   \cite{PAD10}. As a  consequence, the matrix of the transformation between two domains is reflected in the  terms to be estimated. This leads to technical difficulties when fluid is considered to be non-Newtonian, e.g.,  with shear-dependent viscosity obeying             the power-law
$\mu(|e(\vv)|)= \mu{(1+|e( \vv)|^2)^{\frac{p-2}{2}}},\ p>1$.  The application of the Banach fixed point argument for  FSI  problems for power-law fluids is up to authors knowledge an open problem.

\medskip

The main goal of this paper is to show  the convergence of the above described iterative process
for the Newtonian fluid. Here we prove the  essential estimate for the continuous dependence on data and then we apply the Banach fixed point theorem to show the convergence of this iterative method,  which has not yet been done for the fully coupled FSI problems up to our knowledge. As a consequence we also get  uniqueness of the weak solution. 
 We refer to this  iterative process as {\em the global iterative method with respect to the domain deformation} and  use it in our  numerical approximation of the proposed problem, see, e.g., \cite{LZ2008, LZ2010}. In these numerical studies the convergence of these iterations has been observed experimentally.  The present paper provides the first  theoretical proof of this convergence for incompressible Newtonian fluids.
  
The paper is organized as follows. In  Section \ref{sec:linearization} we introduce the linearization of the  problem  (\ref{i.1})--(\ref{i.10}) with respect to the so-called geometric nonlinearity and present known results on the existence of weak solutions on a given deformable domain $\Omega(\delta)$.  In Section \ref{sec:cont_dependence} we derive the estimate on continuous dependence of the weak solution on the boundary pressure and the given deformation. This estimate is formulated in Theorem \ref{th:contin_dependence}. Finally,  in Section \ref{sec:fixed_point} we apply the Banach fixed point theorem and prove the convergence of the global iterative method.
 

}

\section{Linearization with respect to domain geometry $\Omega(h)$}
\label{sec:linearization}
 In order to linearize the problem with respect to the geometric nonlinearity we assume  to have a given deformation function $\delta(x_1,t)$ in the space $H^1(0,T;H_0^2(0,L)) \cap W^{1,\infty}(0,T;L^2(0,L))$. We denote $ h:=R_0+\delta,\ R_0(x_1)\in C^2[0,L]$, and assume
\begin{eqnarray}\label{assumptions}
&&0<\alpha \leq h(x_1,t)\leq \alpha^{-1}, \ \quad
 \left|\frac{\partial h (x_1,t)}{\partial x_1}\right|
+\int_0^T\left|\frac{\partial h(x_1,t)}{\partial t}\right|^2dt
\leq K \qquad \qquad
\end{eqnarray}
for given constants $\alpha>0,\, K>0$. 
We  consider a time dependent domain, which  deforms according to the given function $\delta$, i.e.  we approximate $\Omega(\eta)\approx\Omega(\delta)=:\Omega(h)$.  Due to this approximation the geometric coupling in the original fluid-structure interaction problem  (\ref{i.1})--(\ref{i.10}) has been decoupled.

\subsection{Existence of a weak solution}
\label{sec:existence}
The existence of a  weak solution of the problem (\ref{i.1})--(\ref{i.10}) defined on the given domain with deformation $\delta$, i.e. on $\Omega(h)$ has been shown in \cite[Theorem~5.1]{HLN14} using so called $(\kappa, \veps)$ - approximation, see also   \ref{appendix}.

More precisely, for  boundary pressures\footnote{ In view of the transformation onto $D$ we have $q_{in}(0,y_2,t)=P_{in}(0,x_2,t)/\rho$, $q_{out}(L,y_2,t)=P_{out}(L,x_2,t)/\rho$, $q_{w}(y_1,1,t)=P_{w}(x_1,h(x_1,t),t)/\rho$.}
$q_{in}, q_{out} \in L^\pprime(0,T;L^2(0,1))$ and $q_{w} \in L^\pprime(0,T;L^2(0,L))$  there
exists a weak solution $(\uu, \eta)$, where $\uu$  denotes the transformed fluid velocity $\vv$ into the rectangular domain $D\equiv(0,L)\times(0,1),$ i.e.,
\\ $ \uu(y_1,y_2,t)\stackrel{\text{def}}{=}\vv(y_1,h(y_1,t)y_2,t),\quad y_1=x_1,\, y_2=\frac{x_2}{h(x_1,t)},\, y\in D, x\in \Omega(h),$
with the following properties:
 \begin{itemize}
\item[i)]
 $(\uu,\eta)\in [L^2(0,T;\VV_{\hspace{-1mm}div})\times H^{1}(0,T;{ H^2_0}(0,L))]\cap
[L^{\infty}(0,T;L^2(D)) \times \\ W^{1,\infty}(0,T; L^2 (0,L))]$,
\item[ii)]  the distributive (transformed) time derivative $\bar{\partial}_t(h\uu)\stackrel{\text{def}}{=} \frac{\partial (h {\footnotesize \uu})}{\partial t}
-\frac{1}{h}\frac{\partial h}{\partial t}\frac{\partial (y_2h{\footnotesize \uu} )}{\partial y_2}$
 belongs to 
$ L^\pprime(0,T;X^*) \oplus
L^{4/3}((0,T)\times D)$ 
and \begin{eqnarray*}\label{derivative_hu}
\int_0^T \int_D\left\{  h\uu\cdot \frac{\partial \pp}{\partial t}
+\frac{\partial h}{\partial t}\frac{\partial (y_2\uu )}{\partial y_2}\cdot \pp
\right\}dy\,dt=
-\int_0^T \Big\langle\bar{\partial}_t(h\uu),\pp \Big\rangle_X dt,
\end{eqnarray*}
where
\begin{eqnarray} \label{spaceX}
X&=&\{ \pp \in  \VV_{\hspace{-1mm}div}\, ;
 \   \pp_2\big |_{S_w}\in { H^2_0}(0,L)\},\\ \label{spaceV}
\VV_{\hspace{-1mm}div} &\equiv&\left\{ \ww\in W^{1,2}(D): \dive_h\ww=0 \ a.e.\ \mbox{on}\ D,\right.\\
&&\left.\ w_1=0\hspace{2mm} \mbox{on} \ S_{\it
w}\mbox{ and } w_2=0 \hspace{2mm} \mbox{on} \ S_{\it in}\cup S_{\it out}\cup S_{\it c}   \right\}, \nonumber\\
S_w&=&\{(y_1,1):0<y_1<L \},
\qquad S_{in}=\{(0,y_2):0<y_2<1 \}, \nonumber\\
S_{out}&=&\{(L,y_2):0<y_2<1 \},
\qquad S_c=\{(y_1,0):0<y_1<L \},\nonumber
\end{eqnarray}
\item[iii)] $\uu$ satisfies the transformed divergence-free condition 
\begin{equation}\dive_h  \uu \;\stackrel{\rm def}{=}\;
\frac{\partial u_1}{\partial y_1}
-\frac{y_2}{h}\frac{\partial h}{\partial y_1}\frac{\partial u_1}{\partial y_2}
+\frac{1}{h}
\frac{\partial u_2}{\partial y_2}=0\ \ \mbox{a.e.  on}\  D,\label{divergence}
\end{equation}
\item[iv)] 
$u_2(y_1,1,t)=\partial_t \eta(y_1,t)$ for a.e. $y_1 \in (0,L),\, t \in (0,T)$,\\
 and the following integral identity holds:
 \begin{eqnarray}
 &  0&=\label{weak_form}
\int_0^T\left \langle \bar{\partial_t}(h\uu),\pp \right\rangle_{X}+\big \langle \partial_t  \eta_t , \xi \big \rangle dt +
\int_0^T\hspace{-2mm}\bigg\{\fl \uu , \pp \fr_h +  b_{h}(\uu , \uu , \pp ) \qquad  \\
&&+\int_0^1 h(L) q_{out}(y_2)\psi_1\, (L,y_2)
-h(0) q_{in}(y_2,t)\psi_1\,(0,y_2)\,dy_2 
 \nonumber \\\nonumber
&&+\int_0^L \hspace{-3mm}\big(q_{w}+\frac{1}{2}\frac{\partial h}{\partial t}u_2
\big)\psi_2\,(y_1,1)
\\&&\nonumber \qquad +c\,\frac{\partial^3 \eta}{\partial t \partial y_1^2}\frac{\partial^2\xi}{\partial y_1^2}
+a \frac{\partial \eta }{\partial y_1}\frac{\partial\xi}{\partial y_1}  \nonumber 
-{ a}\frac{\partial^2 R_0}{\partial y_1^2}\xi
+b\eta \xi \,dy_1\hspace{-0mm}\bigg\}\,dt \nonumber
\end{eqnarray}
for every test function
\begin{eqnarray}\label{test_functions}
&& \pp \in
H^1(0,T;\VV_{\hspace{-1mm}div}),
\hspace{2mm}
{\pp_2\big |_{S_w}\in  H^1(0,T;{ H^2_0}(0,L)),\, \pp(y,0)=0,\ \qquad }
\\&&
{  \xi(x_1,t)={ E}\,\pp_2(y_1,1,t)},\ E:={\cal E}\rho.\nonumber
\end{eqnarray}
\end{itemize}
In (\ref{weak_form}) the following notations for the transformed viscous and convective terms have been used:
\begin{eqnarray}\label{eij}
&&\fl \uu,\pp \fr_h:=\frac{\mu}{\rho}\int_D h\eij_{h}(\uu):
\eij_{h}(\pp) dy, \quad 
(\eij_{h}(\uu))_{ij}=\frac{1}{2}(\hat{\partial}_i(u_j)+\hat{\partial}_j(u_i)), \qquad  \\
&&\mbox{where}\ \hat{\partial}_1=\left(\frac{\partial}{\partial y_1}-\frac{y_2}{h}\frac{\partial h}{\partial y_1}\frac{\partial}{\partial y_2}\right ),
\quad \hat{\partial}_2={ \frac{1}{h}\frac{\partial}{\partial y_2} },\nonumber
\end{eqnarray}
 \begin{eqnarray}
 b_h(\uu,\z,\pp)&:= &\int_D
\left( h u_1\left( \frac{\partial \z}{\partial
y_1}-\frac{y_2}{h}\frac{\partial h}{\partial y_1} \frac{\partial \z}{\partial
y_2}\right)+ u_2\frac{\partial \z}{\partial y_2}\right)\cdot \pp
\label{4.3}  
\,dy\hspace{0mm}\\
&&-\frac{1}{2}\int_0^1R_0 u_1 z_1\psi_1\,(L,y_2)\,dy_2
+\frac{1}{2}\int_0^1 R_0 u_1 z_1 \psi_1\,(0,y_2)\,dy_2 \hphantom{ b(\uu,\z}\nonumber \\
&&-\frac{1}{2}\int_0^L u_2 z_2\psi_2\,(y_1,1)\,dy_1.\nonumber
\hspace{-1cm}
\end{eqnarray}
Using integration by parts the following  property of the trilinear form can be proven, see \cite[Lemma 3.7]{HLN14}, 
\begin{eqnarray}
 \label{becko-rozdiel} b_h(\uu, \z, \pp)&=&\int_D \frac{1}{2}B_h(\uu, \z, \pp )-\frac{1}{2}B_h(\uu, \pp, \z)dy \hspace{3mm}\\
\mbox{where}\quad B_h(\uu,\z,\pp )&:=&
\left(hu_1\left(
\frac{\partial \z}{\partial y_1}-\frac{y_2}{h}\frac{\partial h}{\partial y_1}
\frac{\partial \z}{\partial y_2}\right)+
u_2\frac{\partial \z}{\partial y_2}\right)\cdot \pp.\nonumber
\end{eqnarray}
Moreover,  for the distributive time derivative of our weak solution $\uu$ described above holds:
\begin{equation}\label{timederivativeproperty2}
\int_0^T \left \langle \bar{\partial_t}(h\uu),\uu \right\rangle_{X}dt=
\frac{1}{2}\int_D |\uu|^2(T)h(T)-
\frac{1}{2}\int_0^T\int_0^L{\partial_t h}|u_2|^2dy_1 dt, 
\end{equation}
 see \cite[Section 4.2.1 and Appendix A]{HLN14},{ \rred compare the analogous property  in \cite[p. 381]{GRA05}}.

\section{Continuous dependence on the domain deformation}
\label{sec:cont_dependence}
In this section we prove, that the  weak solution is continuously dependent on the boundary pressure data $q_{in/out/w}$ as well as on the domain deformation $h$. The last dependence will be very useful to show the contractivity of the fixed point mapping defined by the iterative process with respect to the domain, that will be studied  in Section~\ref{sec:fixed_point}.

To obtain the continuous dependence estimate 
we use test functions involving the difference of two solutions $\uu^1,\, \uu^2$. Unfortunately, due to the dependency of the divergence operator on $h$, cf. (\ref{divergence}),  the difference $\uu^1- \uu^2$ is not divergence-free. More precisely,  neither $\dive_{h^1}(\uu^1-\uu^2)=0$,  nor $\dive_{h^2}(\uu^1-\uu^2)=0$ in general.
This  is related to the deformability of the domain  and the fact that $\vv^1$ is divergence-free in $\Omega(h^1)$ and $\vv^2$ is divergence-free in $\Omega(h^2)$, i.e., with respect to another coordinates.  In order to design admissible test functions, see Section \ref{sec:admissible}, the first solution is transformed into the definition domain of the second solution.
A similar technique has been previously used  in  \cite{PAD10} to prove the continuous dependence of the weak solution on the initial data  for similar problems in two dimensions using the Eulerian coordinates. In comparison to  \cite{PAD10},  our  problem is also remapped to the fixed reference  domain $D$.
After rewriting the weak formulation by means of  the admissible test functions, see Section \ref{sec:admissible}, multiple error terms arise. Their estimates can be found in Section \ref{sec:estimation}. 
In Section \ref{sec:final_estim} we present the final estimate, which yields the main result of Section \ref{sec:cont_dependence}, 
 the continuous dependence of the weak solution  on the domain deformation $h$ and boundary pressures $q_{in}, \, q_{out},\, q_{w}$ formulated in the following theorem:

\begin{theorem*}[Continuous dependence on data]\ \\
Let $(\uu^1,\eta^1),\, (\uu^2, \eta^2)$ be  two weak solutions of the initial boundary value problem (\ref{i.1})--(\ref{i.10}) transformed to the fixed rectangular  domain $D$ satisfying (\ref{weak_form}). Let the corresponding domain deformation be given  by some functions $h^1, \, h^2$ satisfying (\ref{assumptions}).
Let the transformed boundary pressures  $q_{in/out/w}^1$ and $q_{in/out/w}^2$ belong to $L^2(0,T;L^2(S_{in/out/s}))$, respectively.
Then for almost all $t \in (0,T)$ 
it holds:
\begin{eqnarray*}\label{contin_dependence}
&&\frac{\alpha}{2}\|\RR\uu^1-\uu^2\|_{L^2(D)}^2(t)+\frac{{\alpha\mu}}{2\rho}\tilde{c}_{Ko}\int_0^t\|\RR\uu^1-\uu^2\|_{W^{1,2}(D)}^2ds
  \\&&\nonumber+\frac{E}{2}\|\eta_t^1-\eta_t^2\|_{L^2(0,L)}^2(t)+\frac{bE}{2}\|\eta^1-\eta^2\|_{L^2(0,L)}^2(t)
  +\frac{aE}{2}\|\eta^1_{y_1}-\eta^2_{y_1}\|_{L^2(0,L)}^2(t)
  \\&&\nonumber\ \ +cE\int_0^t\|\eta_{t}^1-\eta_{t}^2\|_{H^2(0,L)}^2ds 
\\&&
\nonumber \leq C_{7}\left(
\int_0^t\|q_{in}^1-q_{in}^2\|^2_{L^2(0,L)}+\|q_{out}^1-q_{out}^2\|^2_{L^2(0,L)}+\|q_{w}^1-q_{w}^2\|^2_{L^2(0,1)} ds \right.
\\&&\hphantom{\leq C_{8}[}\left.
+{ \omega(t)}\left[\nonumber\|h^1-h^2\|^2_{W^{1,\infty}(0,t;L^2(0,L))}+\|h^1-h^2\|^2_{L^\infty(0,t;H^{2} (0,L))}
\right]\right),
\end{eqnarray*}
where \begin{eqnarray*}
\omega(t)=\int_0^t\hspace{-2mm}\|\uu^{1}\|_{W^{1,2}(D)}\hspace{-1mm}+\|\uu^{1}\|^2_{W^{1,2}(D)}\hspace{-1mm}+ \|h^1_t\|_{W^{1,\infty}(0,L)}^2\hspace{-1mm}+\|h^2_t\|_{L^\infty(0,L)}^2
  \hspace{-1mm} +{\rred  \|q^1_{w}\|_{L^2{(S_{w})}}^2 }ds
 \\ \quad \mbox{and} \quad \omega(t)\downarrow 0 \quad  \mbox{for}\quad t\downarrow 0.\hspace{8.4cm}
\end{eqnarray*}
Here $\tilde{c}_{Ko}$ is the coercivity constant of the viscous form coming from the Korn inequality, $\alpha, \, K$ are given by (\ref{assumptions}), $\mu,\rho, E, a, b, c $ are given by the physical model  and  $C_{7}$ is a constant depending on $\alpha,\alpha^{-1},K, \, \tilde{c}_{Ko}^{-1}$ and  on the norms  $\|h^i\|_{L^\infty(0,T;H^2(0,L))}$,  $\|\uu^i\|_{L^\infty(0,T;L^2(D))},$ $i=1,2$.
The matrix $\RR$ in the above estimate arises due to the transformation of weak solution from   $\Omega(h^1)$ to $\Omega(h^2)$. 

\end{theorem*}


\subsection{Admissible test functions}
\label{sec:admissible} 
To prove the above theorem let us assume  $(\uu^1,\eta^1), (\uu^2,\eta^2)$ are two weak solutions of our problem transformed 
onto a fixed reference domain $D$, both satisfying the  weak formulation (\ref{weak_form}) with $h^1,\, h^2$, respectively.  


Now we provide some preliminaries concerning admissible test functions (\ref{test_functions}), involving a difference of the two weak solutions.
Let us  denote the matrix of the  transformation of variables between $\Omega(h^1)$ and $\Omega(h^2)$ by  $\JJ=\frac{d X}{dx}$,  $X\in\Omega(h^1)$, $x\in\Omega(h^2)$    
 and the inverse matrix by $\JJ^{-1}$. In the reference coordinates $y\in D$ they read,
\begin{eqnarray*}
\JJ=\begin{bmatrix}1 & 0\\y_2 h^2\partial_{ y_1}(\frac{h^1}{h^2})& \frac{h^1}{h^2}\end{bmatrix},
\quad \JJ^{-1}=\begin{bmatrix}1 & 0\\y_2{h^1}\partial_{ y_1}(\frac{h^2}{h^1})& \frac{h^2}{h^1}\end{bmatrix}.
\end{eqnarray*} 
Let  $\RR:= J \JJ^{-1}$, where  $J=|\JJ|$ is the determinant of the transformation matrix $\JJ$. We have  $J=h^1/h^2$ and
\begin{eqnarray}\label{matrix_R}
\RR=J \JJ^{-1}=\begin{bmatrix}\frac{h^1}{h^2} & 0\\-y_2 h^2\partial_{ y_1}(\frac{h^1}{h^2})& 1\end{bmatrix}, \quad \RR^{-1}=\begin{bmatrix}\frac{h_2}{h_1} & 0\\-y_2 h^1\partial_{ y_1}(\frac{h^2}{h^1})& 1\end{bmatrix}.
\end{eqnarray} 
Next step is to design admissible divergence-free test function according to (\ref{test_functions}), that contain a difference of the two weak solutions.  To this end we consider two sets of test functions for the fluid velocity
\begin{eqnarray}
\pp^1=\uu^1-\RR^{-1}\uu^2,   \qquad \label{pps}
\pp^2=\RR\uu^1-\uu^2, 
\\ \mbox{i.e.,}\ \qquad \pp^1=\RR^{-1}\pp^2,\qquad \qquad \   \pp^2=\RR\pp^1. \qquad\   \nonumber
\end{eqnarray}  
It holds that $\dive_{h^2}\RR\uu^1=0$.  Indeed, since $\RR\uu^1=\big[\frac{h^1}{h^2}u^1_1,-y_2h^2\partial_{y_1}(\frac{h^1}{h^2}) u^1_1+u^1_2\big]^T,$ we have
 \begin{eqnarray*}
\dive_{h^2}\RR\uu^1&=&\frac{\partial}{\partial y_1}\left(\hfrac u^1_1\right)-\frac{y_2}{h^2}\partial_{y_1}h^2\frac{\partial}{\partial y_2}\left(\hfrac u^1_1\right)+\frac{1}{h^2}\frac{\partial}{\partial y_2}\left(u^1_2-y_2h^2\partial_{y_1}\big(\hfrac\big)u^1_1 \right)\\
&=&\hfrac\partial_{y_1}u^1_1 - \frac{y_2}{h^2}\partial_{y_1}h^2\hfrac\partial_{y_2}u^1_1  + \frac{1}{h^2}\partial_{y_2}u^1_2 - y_2\partial_{y_1}\big(\hfrac\big)\partial_{y_2}u^1_1\\
&=& \hfrac \partial_{y_1}u^1_1 + \frac{1}{h^2}\partial_{y_2}u^1_2  -   
\frac{y_2}{h^2}\partial_{y_2}u^1_1 \partial_{y_1}h^1=\hfrac \dive_{h^1}\uu^1,   
\end{eqnarray*}   cf. (\ref{divergence}).
Thus we get $\dive_{h^2}\RR\uu^1=J\dive_{h^1}\uu^1\equiv0$. Analogously one can  show $\dive_{h^1}\RR^{-1}\uu^2=J^{-1}\dive_{h^2}\uu^2=0$.
Thus $$\dive_{h^1}\pp^1=\dive_{h^2}\pp^2= 0,$$ 
i.e., $\pp^i$ is solenoidal with respect to $\Omega(h^i)$  for  $i=1,2.$
Moreover,  since the first component  of the velocity  $u^{i}_1|_{S_w}=0$  we have $\RR\uu^1|_{S_w}=[0,u^1_2]^T$. Similarly $\RR^{-1}\uu^2|_{S_w}=[0,u^2_2]^T.$ Thus   the trace of  $\pp^i$ on ${S_w}$ is equal to $[0,u_2^1-u^2_2]^T$ for both $i=1,2$, 
and  due to  $u^{i}_2|_{S_w}=\eta_t^i$ a.e., we  obtain $\psi^i_2|_{S_w}=\eta_t^1-\eta_t^2$. Consequently, the test functions $(\pp^1,\xi)$,  $(\pp^2,\xi)$ from (\ref{pps}) with  $ \xi= E\eta_t^1-E \eta_t^2$ are admissible.

\medskip

\subsection{Subtraction of weak formulations}
{\rred In order to obtain the final estimate let us first pay attention to the time derivative terms in the  weak formulation (\ref{weak_form}). Since 
$\pp^2=\RR \uu^1-\uu^2$ and due to the linearity of $\bar{\partial}_t^h$ we can write:
\begin{eqnarray}\label{diff_timeder}
&&\left\langle\bar{\partial}_t^{h^2}(h^2\pp^2),\pp^2 \right\rangle_X=\left\langle\bar{\partial}_t^{h^2}(h^2\RR\uu^1),\pp^2 \right\rangle_X-\left\langle\bar{\partial}_t^{h^2}(h^2\uu^2),\pp^2 \right\rangle_X.
\end{eqnarray}



The term on the left-hand side  can rewritten using
the  property (\ref{timederivativeproperty2}) of the distributive time derivative. 
Note that due to the coupling between the velocity $\uu^1$ and the domain deformation $h^1$ the first functional  on the right-hand  side of (\ref{diff_timeder}) is not yet defined through (\ref{weak_form}) as it is the case of the second term. Indeed,  we  cannot assert that $\RR \uu^1$ is a weak solution associated with $h^2$. Therefore, we have to investigate  the object $\bar{\partial}_t^{h^2}(h^2\RR\uu^1)$.
}
To this end we derive the following equality
\begin{equation}\label{ALE_timeder}
\bar{\partial}_t^{h^2}(h^2\RR\uu^1)=\JJ^{-1}\bar{\partial}_t^{h^1}(h^1\uu^1)+\EE_1,
\end{equation} where the additional error term $\EE_1$ reads
\begin{eqnarray}\label{matrix_E1}
\EE_1=
\begin{bmatrix}
[\frac{h^1_t}{h^1}-\frac{h^2_t}{h^2}] \partial_{y_2}(y_2 h^1u^1_1)
\\
[ \frac{h^1_t}{h^1}-\frac{h^2_t}{h^2}] y_2h^2\left( \frac{\partial u^1_2}{\partial y_2}-y_2 E
\frac{\partial u^1_1}{\partial y_2}\right)+y_2 
u^1_1 [ \partial_t h^2 E-h^2\partial_t E ]
\end{bmatrix}
\end{eqnarray}
with $E:=\frac{h^1_{y_1} h^2-h^1 h^2_{y_1} }{h^2}=h^2\partial_{y_1}({\hfrac})$.
To show (\ref{ALE_timeder}), (\ref{matrix_E1}) tedious but straightforward manipulations have been performed as follows.
\\For the first component  of $h^2\RR\uu^1$ we have
\begin{eqnarray*}
\bar{\partial}_t^{h^2}(h^2\RR\uu^1)_1&=&
\partial_t(h^1 u^1_1)-\frac{1}{h^2}h^2_t\partial_{y_2}(y_2 h^1 u^1_1)\\
&=&\bar{\partial_t}^{h^1}(h^1 u^1_1)+\partial_{y_2}(y_2 h^1 u^1_1) \left[\frac{h^1_t}{h^1}-\frac{h^2_t}{h^2}\right ], 
\end{eqnarray*}
compare the definition of the distributive time derivative in Section \ref{sec:existence}. 
\\Since $(\RR\uu^1)_2=u^1_2-y_2Eu_1^1$,  we can write for the second component
\begin{eqnarray*}
\bar{\partial}_t^{h^2}(h^2\RR\uu^1)_2&=&
\bar{\partial_t}^{h^2}(h^2u_2^1)-\bar{\partial_t}^{h^2}(h^2y_2E u^1_1)\\
&=&\bar{\partial_t}^{h^2}(h^2u_2^1)-\partial_t(y_2h^2E u^1_1)+\frac1{h^2}\frac{\partial h^2}{\partial t}\frac{\partial(y_2^2 h^2 E u^1_1) }{\partial {y_2}} \\
&=&\bar{\partial_t}^{h^2}(h^2u_2^1)-y_2E\bar{\partial_t}^{h^2}(h^2u_1^1)-y_2u^1_1(h^2\partial_t E-E\partial_t h^2).
\end{eqnarray*}
Using the fact that $ \bar{\partial_t}^{h^2}(h^2u_i^1)= \frac{h^2}{h^1}\bar{\partial_t}^{h^1}(h^1u_i^1)+y_2h^2\partial_{y_2} u^1_i[\frac{h^1_t}{h^1}-\frac{h^2_t}{h^2}],\, i=1,2,$ and replacing $-y_2E$ by $y_2\hfrac h^1\partial_{y_1}(\hfraci)$ in the second term on the right-hand side of last equation, we finally arrive at 
\begin{eqnarray*}
\bar{\partial}_t^{h^2}(h^2\RR\uu^1)_2&=&\frac{h^2}{h^1}\bar{\partial_t}^{h^1}(h^1u_2^1)+y_2 h^1 \partial_{y_1}\left(\frac{h^2}{h^1}\right)\bar{\partial_t}^{h^1}(h^1u^1_{1})
\\&& + \left[\frac{h^1_t}{h^1}-\frac{h^2_t}{h^2}\right]y_2h^2\left( \frac{\partial u^1_2}{\partial y_2}-y_2 E
\frac{\partial u^1_1}{\partial y_2}\right)+y_2 
u^1_1 [ \partial_t h^2 E-h^2\partial_t E ].
\end{eqnarray*}
Thus (\ref{ALE_timeder}), (\ref{matrix_E1}) hold true, see also the definition of the matrix $\JJ^{-1}$ above.

Finally, using (\ref{timederivativeproperty2}) and (\ref{ALE_timeder}), (\ref{matrix_E1}) and   $\RR\pp^1= \pp^2$   we obtain from (\ref{diff_timeder})
\begin{eqnarray*}
&&\int_0^T\left\langle\JJ^{-1}\bar{\partial}_t^{h^1}(h^1\uu^1),\RR\pp^1 \right\rangle_X-\left\langle\bar{\partial}_t^{h^2}(h^2\uu^2),\pp^2 \right\rangle_X+\int_D \EE_1 \cdot \pp^2 dy\,  dt\quad  \\\nonumber
&&\qquad 
 = \frac{1}{2}\int_D |\pp^2|^2(T)h^2(T)dy
  - \frac{1}{2}\int_0^T\int_0^L \partial_th^2|\psi_2^2|^2 dy_1 dt.
\end{eqnarray*}
Replacing $\left\langle\JJ^{-1}\bar{\partial}_t^{h^1}(h^1\uu^1),\RR\pp^1 \right\rangle_X$ by  $\left\langle\RR^T\JJ^{-1}\bar{\partial}_t^{h^1}(h^1\uu^1),\pp^1 \right\rangle_X$  we obtain
\begin{eqnarray}\label{timeder_diff}
  \frac{1}{2}\int_D |\pp^2|^2(T)h^2(T)dy  dt&\hspace{-2mm}=&\hspace{-2mm}\int_0^T\hspace{-2mm}\left\langle\RR^T\JJ^{-1}\bar{\partial}_t^{h^1}(h^1\uu^1),\pp^1 \right\rangle_X-\left\langle\bar{\partial}_t^{h^2}(h^2\uu^2),\pp^2 \right\rangle_X\hspace{-1mm} dt
\qquad    \\\nonumber&& + \frac{1}{2}\int_0^T\hspace{-2mm}\int_0^L \partial_th^2|\psi_2^2|^2 dy_1 
 +\int_D \EE_1 \cdot \pp^2 dy\,  dt.
\end{eqnarray}
{\rred In order to obtain  the continuous dependence  estimate,  one needs to replace the pairings on  the right-hand side of (\ref{timeder_diff})   according to the associated  weak formulations. However, concerning the first pairing, the weak formulation for $\uu^1,\eta^1 $ has to be  multiplied with the  matrix $\RR^T\JJ^{-1}=\hfrac\JJ^{-T}\JJ^{-1} $.
  Note that $\|\RR^T\JJ^{-1}\|_\infty$ is bounded from above by some constant dependent on $\alpha, K$, cf. (\ref{assumptions}). Moreover this matrix is symmetric and positive definite with (strongly) positive eigenvalues. Thus the existence of a weak solution $\uu^1,\eta^1 $ for  the transformed weak formulation associated with pairing $\left\langle\RR^T\JJ^{-1}\bar{\partial}_t^{h^1}(h^1\uu^1),\pp^1 \right\rangle_X$ can be derived using the same technique as for the original weak formulation associated with  $\left\langle\bar{\partial}_t^{h^1}(h^1\uu^1),\pp^1 \right\rangle_X$.
Indeed,  for the sake of brevity we point out here only the coercivity of the transformed viscous form 
$\int_Dh^1 \RR^T\JJ^{-1}\eij_{h^1} (\uu^1):\eij_{h^1} (\uu^1)dy\equiv \int_D h^1\eij_{h^1}(\uu^1)\RR^T\JJ^{-1} \eij_{h^1} (\uu^1)dy$. After the orthogonal diagonalization of $\RR^{T} \JJ^{-1}$ with eigenvalues $\lambda_{1,2}>0$, this viscous form can be bounded from below with $\int_D\alpha \lambda_{\min} |\eij_{h^1} (\uu^1)|^2dy $, 
consequently the generalized Korn inequality, cf. \cite{PO03}, can be applied, see also lines following (\ref{weak_form_diff2}).

Now we pay attention about the viscous, convective and boundary terms coming from the replacement of the first pairing on the right-hand side  of (\ref{timeder_diff}) according the associated transformed weak formulation. 
In analogy to (\ref{ALE_timeder}) 
  we rewrite them by means of the deformation function $h^2$.
}
\medskip
\\{\bf  Viscous terms:}\\
The following lemma will be  useful to handle the viscous terms.

\begin{lemma} \label{lemma:viscous} For the  deformation tensor $\eij_h$, see (\ref{eij}), it holds
\begin{equation}\label{visc_trans}
\eij_{h^1} (\vv)=\eij_{h^2} (\RR \vv)-[\EE(\vv)+\EE(\vv)^T],
\end{equation}
\begin{eqnarray}
\mbox{where }&&\EE(\vv)=\EE_2(\vv)F_{h^1}+\EE_2(\vv)\EE_3+\nabla\vv \EE_3
\label{matrix_E}
\\ \mbox{and  }&&\EE_2(\vv)=\begin{bmatrix}
v_1 \frac{E}{h^2}+\frac{h^1-h^2}{h^2}\partial_{y_1}v_1  & \frac{h^1-h^2}{h^2}\partial_{y_2}v_1\\
-y_2\partial_{y_1}(E v_1) & -E\partial_{y_2}(y_2v_1)
\end{bmatrix}\hspace{-1.5mm},\    E=h^2 \partial_{y^1}(\hfrac),\qquad \label{matrix_E2}
\\ && \EE_3=\frac{1}{2}\begin{bmatrix}
0 & 0\\
y_2 \frac{E}{h^1} & [\frac{1}{h^1}-\frac{1}{h^2}]
\end{bmatrix}, \quad 
 F_{h^1}=\frac{1}{2}\begin{bmatrix}
1 & 0\\
-\frac{y_2}{h^1} \partial_{y_1}{h^1} & \frac{1}{h^1}
\end{bmatrix}.\label{matrix_E3}
\end{eqnarray}
\end{lemma}
\begin{proof}
Taking into account the definition (\ref{matrix_R}) of the transformation matrix $\RR$, 
 one can easily obtain 
\begin{eqnarray}\label{nabla_rv}
\nabla (\RR \vv)=\nabla \vv+ 
\begin{bmatrix} \partial_{y_1}(\hfrac)v_1+\frac{h^1-h^2}{h^2}\frac{\partial v_1}{\partial y_1} & \frac{h^1-h^2}{h^2}\frac{\partial v_1}{\partial y_2}\\
-y_2\partial_{y_1}(h^2 \partial_{y_1}(\hfrac))v_1 & -h^2 \partial_{y_1}(\hfrac) \frac{\partial ( y_2 v_1)}{\partial y_2}
\end{bmatrix} =\nabla \vv+\EE_2(\vv).\ 
\end{eqnarray}
Moreover, it is easy to show that
\begin{equation}\label{def_tensor}
  \eij_{h}(\vv)=\nabla \vv F_{h}+(\nabla \vv F_{h})^T,
\end{equation}
see, e.g., \cite[Proof of Lemma 3.4]{HLN14}. Therefore 
we can write
\begin{eqnarray*}\eij_{h^2}(\RR \vv)=\nabla (\RR \vv)F_{h^2}+(\nabla \RR \vv F_{h^2})^T=( \nabla \vv+\EE_2(\vv))F_{h^2}+F_{h^2}^T( \nabla \vv^T+\EE_2(\vv)^T).
\end{eqnarray*}
One can  verify that $F_{h^2}=F_{h^1}+\EE_3$. Inserting this into  above equality we obtain
\begin{eqnarray*}
\eij_{h^2}(\RR \vv)&=&
\eij_{h^1}(\vv)+\EE_2(\vv)F_{h^1}+F_{h^1}^T\EE_2(\vv)^T+ 
\EE_2(\vv)\EE_3+\EE_3^T\EE_2(\vv)^T\\&&+ \nabla \vv\EE_3+\EE_3^T\nabla \vv^T, 
\end{eqnarray*} which proves (\ref{visc_trans}).
\end{proof}
{\rred Now, for the difference of the viscous terms on the right-hand side of  (\ref{timeder_diff}) 
  using $\RR=(\II+\EE_{\RR})$ where $(\EE_{\RR})_{11}=\frac{h^1-h^2}{h^2},(\EE_{\RR})_{21} =-y_2h^2\partial_{y_1}(\hfrac),(\EE_{\RR})_{12}=(\EE_{\RR})_{22}=0$ (see (\ref{matrix_R}))  and   Lemma \ref{lemma:viscous},  one has
\begin{eqnarray}\label{viscous_diff}
  &&\frac{\mu}{\rho}\int_0^T\int_D h^1\RR^{T}\JJ^{-1} \eij_{h^1}(\uu^1):\eij_{h^1}(\pp^1)- h^2 \eij_{h^2}(\uu^2):\eij_{h^2}(\pp^2)dy\,  dt
  \\&&=\frac{\mu}{\rho}\int_0^T\int_D h^2\RR \eij_{h^1}(\uu^1):\RR\eij_{h^1}(\pp^1)- h^2 \eij_{h^2}(\uu^2):\eij_{h^2}(\pp^2)dy\,  dt\nonumber
  \\&&=\nonumber\frac{\mu}{\rho}\int_0^T\hspace{-2mm}\int_D h^2\Big[\eij_{h^1}(\uu^1):\eij_{h^1}(\pp^1)  +\EE_{\RR}\eij_{h^1}(\uu^1):\RR\eij_{h^1}(\pp^1)+\eij_{h^1}(\uu^1):\EE_{\RR}\eij_{h^1}(\pp^1)
  \\&&\hphantom{\int_0^T\int_D h^2[}
  -\eij_{h^2}(\uu^2):\eij_{h^2}(\pp^2)\Big] dy dt\nonumber
\\&&\stackrel{(\ref{visc_trans}) }{=}\frac{\mu}{\rho}\int_0^T\hspace{-2mm}\int_D\hspace{-2mm} h^2\Big[
\left(\eij_{h^2}(\RR \uu^1)-\EE(\uu^1)-\EE(\uu^1)^T\right)\hspace{-1mm}:\hspace{-1mm}
\left(\eij_{h^2}(\RR \pp^1)-\EE(\pp^1)-\EE(\pp^1)^T\right)
\nonumber\\&& \hphantom{\int_0^T\int_D h^2[}
  +(\RR^T\EE_{\RR}+\EE_{\RR}^T)\eij_{h^1}(\uu^1):\eij_{h^1}(\pp^1)
  -\eij_{h^2}(\uu^2):\eij_{h^2}(\pp^2)\Big]
dy\,  dt
\nonumber\\&&=\frac{\mu}{\rho}\int_0^T\hspace{-2mm}\int_D h^2\Big[\eij_{h^2}(\pp^2):\eij_{h^2}(\pp^2)-\underbrace{(\EE(\uu^1)+\EE(\uu^1)^T):\eij_{h^2}(\pp^2)}_{:=I_1}
\nonumber\\&& \hphantom{\int_0^T}-\underbrace{\eij_{h^1}(\uu^1):(\EE(\pp^1)+\EE(\pp^1)^T)}_{:=I_2}
+\underbrace{(\RR^T\EE_{\RR}+\EE_{\RR}^T)\eij_{h^1}(\uu^1):\eij_{h^1}(\pp^1)}_{:=I_3}\Big]dy\,  dt.\nonumber
\end{eqnarray}
}
{\bf  Convective terms:}

\noindent{\rred
Recalling (\ref{becko-rozdiel}), we first rewrite the trilinear term using matrix the $F_h$ in  (\ref{matrix_E3}) as $$B_h(\uu,\z,\pp)=h(\uu\cdot F_h^T\nabla)\z\cdot \pp.$$ Thus  for the corresponding convective term from the first pairing on the right-hand side of  (\ref{timeder_diff}) holds due to the fact that $\RR\pp^1=\pp^2$, $h^1\JJ^{-1}=h^2\RR$  and  $F_{h^2}^{-1} F_{h^1}=\frac{h^2}{h^1}\RR$:
\begin{eqnarray*}
  && \int_D \RR^T\JJ^{-1} B_{h^1}(\uu^1,\uu^1,\pp^1)dy=\int_Dh^1\RR^T\JJ^{-1}(\uu^1\cdot F_{h^1}^T\nabla)\uu^1\cdot \pp^1 dy
  \\&&=
  \int_Dh^2\RR(\uu^1\cdot F_{h^1}^T\nabla)\uu^1\cdot \pp^2 dy
  =\int_Dh^2\RR(\uu^1\cdot\hfraci\RR^T F_{h^2}^T\nabla)\uu^1\cdot \pp^2 dy
  \\&&=\int_Dh^2\hfraci\RR\left(\RR\uu^1\cdot F_{h^2}^T\nabla\right)\uu^1\cdot \pp^2 dy
  \\&&=\int_DJ^{-1}\underbrace{h^2\left(\RR\uu^1\cdot F_{h^2}^T\nabla\right)(\RR\uu^1)\cdot \pp^2}_{}
  -J^{-1}\underbrace{h^2\left(\RR\uu^1\cdot F_{h^2}^T\nabla\right)(\RR)\uu^1\cdot \pp^2}_{} dy,
  \\&& =\int_D\quad \ J^{-1}\ B_{h^2}(\RR\uu^1, \RR\uu^1,\pp^2)\qquad  -  \quad J^{-1} T_{h^2}(\RR\uu^1,\uu^1,\pp^2) dy,
\end{eqnarray*}
where we recall $J=\hfrac$, see (\ref{matrix_R}), and $ T_{h}(\RR\uu,\vv,\pp):=h\left(\RR\uu\cdot F_{h}^T\nabla\right)(\RR)\vv\cdot \pp$.
\\Analogously for the second part of the convective term one can obtain
\begin{eqnarray*}
   \int_D \RR^T\JJ^{-1} B_{h^1}(\uu^1,\pp^1,\uu^1)
   =\int_D\hspace{-2mm}J^{-1} B_{h^2}(\RR\uu^1, \pp^2, \RR\uu^1) -  J^{-1} T_{h^2}(\RR\uu^1,\pp^1,\RR\uu^1) dy.
\end{eqnarray*}
Using the equalities above and the trilinearity of $B_{h^2}$, 
one gets for the difference
\begin{eqnarray}\label{becko1}
 &&  \int_D \RR^T\JJ^{-1} B_{h^1}(\uu^1,\uu^1,\pp^1)- B_{h^2}(\uu^2,\uu^2,\pp^2)dy
   \\\nonumber&&=\int_D\hspace{-2mm} B_{h^2}(\RR\uu^1,  \RR\uu^1,\pp^2) -  B_{h^2}(\uu^2,\uu^2,\pp^2)+ \hfracdi B_{h^2}(\RR\uu^1,  \RR\uu^1,\pp^2)
   \\\nonumber&&\qquad  -\hfraci T_{h^2}(\RR\uu^1,\uu^1,\pp^2) dy
   \\\nonumber&&=\int_D\hspace{-2mm} B_{h^2}(\pp^2,  \RR\uu^1,\pp^2) + B_{h^2}(\uu^2,\pp^2,\pp^2)+ \hfracdi B_{h^2}(\RR\uu^1,  \RR\uu^1,\pp^2)
 \\ \nonumber&&\qquad  -\hfraci T_{h^2}(\RR\uu^1,\uu^1,\pp^2) dy,
\end{eqnarray}
Analogously one can obtain for the difference of the second part of  convective term
\begin{eqnarray}\label{becko2}
 &&  \int_D \RR^T\JJ^{-1} B_{h^1}(\uu^1,\pp^1,\uu^1)- B_{h^2}(\uu^2,\pp^2,\uu^2)dy
   \\\nonumber&&=\int_D\hspace{-2mm} B_{h^2}(\pp^2,  \pp^2,\RR\uu^1)   + B_{h^2}(\uu^2,\pp^2,\pp^2)+\hfracdi B_{h^2}(\RR\uu^1, \pp^2 ,\RR\uu^1)
   \\\nonumber&& \qquad -\hfraci T_{h^2}(\RR\uu^1,\pp^1,\RR\uu^1) dy.
\end{eqnarray}

The difference of the convective  terms on the right-hand side of  (\ref{timeder_diff})
can be obtain by subtracting (\ref{becko1}) and (\ref{becko2}).  Using  (\ref{becko-rozdiel}) we can finally write 
\begin{eqnarray}\label{convective_diff}
  &&\int_0^T \frac1{2}\int_D \RR^T\JJ^{-1} [B_{h^1}(\uu^1,\uu^1,\pp^1)- B_{h^1}(\uu^1,\pp^1,\uu^1)]dy
  -b_{h^2}(\uu^2,\uu^2,\pp^2) dt\qquad 
\\\nonumber&&=\int_0^T b_{h^2}(\pp^2,\RR\uu^1,\pp^2)+ \hfracdi b_{h^2}(\RR\uu^1,\RR\uu^1,\pp^2)
\\&& \hphantom{\int_0^T} -\frac1{2}\int_D \hfraci[ T_{h^2}(\RR\uu^1,\uu^1,\pp^2)-T_{h^2}(\RR\uu^1,\pp^1,\RR\uu^1)]dy dt.\nonumber
\end{eqnarray}
}
{\bf  Boundary  terms:} \\Note that due to the boundary conditions on $S_w$,  the first component $(\pp^i(y_1,1,t))_1=( \uu^i(y_1,1,t))_1=0$, $i=1,2$. Recalling (\ref{pps}),  i.e., $(\pp^2)_2=(\RR\pp^1)_2=\RR_{22}\psi^1_2 $ and $\RR_{22}=1,\, \JJ^{-1}_{22}=\hfraci$, for the difference of boundary terms we can  write 
\begin{eqnarray}\label{bnd1}
&&\frac{1}{2}\int_0^T\int_0^L(\partial_t h^1\RR^T_{22}\JJ^{-1}_{22}u^1_2 \psi^1_2 -\partial_t h^2u^2_2 \psi^2_2)(y_1,1,t)dy_1 dt\\
&&=\frac{1}{2}\int_0^T\int_0^L\partial_t h^2[\JJ^{-1}_{22}u^1_2 \RR_{22}\psi^1_2-u^2_2 \psi^2_2]+\JJ^{-1}_{22}u^1_2 \RR_{22}\psi^1_2[\partial_t h^1-\partial_t h^2](y_1,1,t)dy_1 dt\nonumber
\\&&=\frac{1}{2}\int_0^T\int_0^L
\partial_t h^2\left[(\hfraci u^1_2-u^2_2) \psi^2_2
  \right]\nonumber
  +\hfraci u^1_2 \psi^2_2\partial_t [h^1- h^2](y_1,1,t)dy_1 dt 
 \\&& =\frac{1}{2}\int_0^T\hspace{-1mm}\int_0^L \hspace{-2mm}\partial_t h^2 |\psi^2_2|^2+ \left(\hfraci\partial_t [h^1- h^2]+\hfracdi{h^2_t}\right)u^1_2 \psi^2_2(y_1,1,t)dy_1 dt.\nonumber
\end{eqnarray}
Let us mention that the first term on the right-hand side of (\ref{bnd1}) 
vanishes in the sum with  the boundary term arising on the right-hand side of (\ref{timeder_diff}). Moreover, the pressure terms on $S_w$ can be simplified to 
\begin{eqnarray}\label{bnd2}
&&\hspace{-5mm}\int_0^T\int_0^L(\RR^T_{22}\JJ^{-1}_{22}q_{w}^1\psi^1_2 - q_{w}^2 \psi^2_2)(y_1,1,t)dy_1 dt
  \\&&\hspace{-5mm}\nonumber =\int_0^T\hspace{-2mm}\int_0^L\hspace{-1mm}(\hfraci q_{w}^1- q_{w}^2)\psi^2_2(y_1,1,t)
  dy_1 dt =\int_0^T\hspace{-2mm}\int_0^L\hspace{-2mm}(q_{w}^1- q_{w}^2)\psi^2_2+\hfracdi q_w^1\psi^2_2(y_1,1,t).
\end{eqnarray}
Concerning the boundary pressure terms on $S_{in/out}$ note that the second component of the velocity is zero, thus $(\pp^2)_1=(\RR\pp^1)_1=\RR_{11}\psi^1_1$, 
and $\JJ^{-1}_{11}=1$. Thus
\begin{eqnarray}\label{bnd3}
 && \hspace{-5mm}\int_0^T\hspace{-2mm}\int_0^1\hspace{-1mm}(q_{in}^1h^1\RR^T_{11}\JJ^{-1}_{11}\psi^1_1 - q_{in}^2h^2 \psi^2_1)(0,y_2,t)dy_2 dt=
  \\&&\hspace{-5mm}\nonumber
  \int_0^T\hspace{-2mm}\int_0^1\hspace{-1mm}(q_{in}^1h^1- q_{in}^2h^2) \psi^2_1
=\int_0^T\hspace{-2mm}\int_0^1\hspace{-1mm}[h^1(q_{in}^1- q_{in}^2)+(h^1-h^2)q_{in}^2 ]\psi^2_1(0,y_2,t)dy_2 dt.
\end{eqnarray}  In the same way the corresponding difference involving $q_{out}$ can be obtained.

\medskip
Finally, after replacing the {\rred difference of pairings on the right-hand side of  (\ref{timeder_diff}) with differences
of viscous terms, convective and boundary terms} 
according to (\ref{viscous_diff}), (\ref{convective_diff}) and (\ref{bnd1})--(\ref{bnd3}), respectively, the resulting equation 
reads
\begin{eqnarray}\label{weak_form_diff2}
&& \frac{1}{2}\int_D|\pp^2|^2(T)h^2(T) dy
+\frac{\mu}{\rho}\int_0^T\int_D h^2\eij_{h^2}(\pp^2):\eij_{h^2}(\pp^2)dy dt
\\\nonumber&&\hspace{0mm}+\frac{1}{2E}\int_0^L|\xi|^2(T)
+\frac{aE}{2}\left|\partial_{y_1}\bar{\eta}\right|^2(T)+\frac{bE}{2}|\bar{\eta}|^2 (T)
+\frac{c}{E}\int_0^T\int_0^L|\partial^2_{y_1}\xi|^2dy_1 dt\nonumber 
\\&&=
\int_0^T 
\int_D \EE_1 \cdot \pp^2 + \frac{\mu}{\rho}h^2(I_1+I_2)-I_3 \, dy\,  dt\nonumber
\\&&
-\int_0^T  
{\rred b_{h^2}(\pp^2,\RR\uu^1,\pp^2)+\hfracdi b_{h^2}(\RR\uu^1,\RR\uu^1,\pp^2)\, dt\nonumber}
\\&& {\rred+\int_0^T\frac1{2}\int_D \hfraci[ T_{h^2}(\RR\uu^1,\uu^1,\pp^2)-T_{h^2}(\RR\uu^1,\pp^1,\RR\uu^1)]dy \,dt.}\nonumber
\\&&\rred \int_0^T\int_0^1-(q_{in}^1- q_{in}^2)h^1 \psi^2_1(0,y_2,t)\nonumber
  +(q_{out}^1- q_{out}^2)h^1 \psi^2_1(L,y_2,t)dy_2  \nonumber 
\\&&\rred-\hspace{-2mm}\int_0^L\left[(q_{w}^1- q_{w}^2)+\hfracdi q_w^1  + \frac{1}{2} \left(\hfraci\partial_t [h^1- h^2]+\hfracdi{h^2_t}\right)u^1_2
\right]\psi^2_2(y_1,1,t) dy_1 dt,
\nonumber 
\end{eqnarray}
where the error terms $\EE_1,\,I_1,\, I_2,\, I_3$ are defined in 
(\ref{matrix_E1}), (\ref {matrix_E}),  (\ref {matrix_E2}),   (\ref {matrix_E3}), (\ref{viscous_diff}),
respectively, and $T_h(\uu,\vv,\pp):=h(\uu\cdot F_h^T\nabla)(\RR)\vv\cdot\pp$.

To derive an appropriate  estimate we apply  Korn's first inequality with variable coefficients, see, e.g., \cite{NEFF02}, \cite{PO03} in the second term on the left-hand side of the above equality. Indeed, using (\ref{def_tensor}) we can write
$$\frac{\mu}{\rho}\int_0^T\int_D h^2|\eij_{h^2}(\pp^2)|^2dy dt=\frac{\mu}{\rho}\int_0^T\int_D h^2|\nabla \pp^2 F_{h^2}+(\nabla \pp^2 F_{h^2})^T|^2,$$ where $F_{h^2}=F_{h^2}(y):\bar{\Omega}\to \RR^{2\times 2}$ is a continuous mapping with  $det(F_h^2)=\frac{1}{h^2}>\alpha>0$, cf. (\ref{matrix_E3}). Consequently, Corollary 4.1 in \cite{PO03} implies the existence of  a positive constant $c_{Ko}$ such that
$$\frac{\mu}{\rho}\int_D h^2|\eij_{h^2}(\pp^2)|^2dy\geq c_{Ko}\frac{\mu}{\rho}\alpha\int_D| \nabla\pp^2|^2dy,\ \forall \pp^2\in \VV_{\hspace{-1mm}div}. $$


 Next we will estimate the terms on the right-hand side of (\ref{weak_form_diff2})  in an appropriate way to apply the Gronwall lemma.

 \subsection{Estimate of the right-hand side of (\ref{weak_form_diff2})}
\label{sec:estimation}
In what follows we use the notation $\bar{h}=h^1-h^2$. We also express the terms with $E=h^2\partial_{y_1}(\hfrac)$ by means of the differences $\barh$, 
\begin{equation}\label{barh_E}
h^2 E=
h^2\partial_{y_1}h^1 -h^1\partial_{y_1}h^2= 
h^1\partial_{y_1}\bar{h}-\partial_{y_1}h^1\bar{h}.
\end{equation}
Before we start presenting the estimates, we introduce the following useful lemma.
\begin{lemma}\label{lemma:essup}
For all $\vv$ in the space $\VV_{\hspace{-1mm}div}$ (defined in  (\ref{spaceV})), it holds
$$\mbox{ ess\hspace{-3mm}}\sup_{0\leq y_1\leq L\ } \int_0^1|\vv|^2 dy_2\leq c \|\vv\|_{L^2(D)} \|\nabla \vv\|_{L^2(D)}. $$
\end{lemma}
\begin{proof}
  We can write
  \begin{eqnarray*}
    \int_0^1|\vv(y_1,y_2,t)|^2dy_2&=& \int_0^1|\vv(0,y_2,t)|^2dy_2+\int_0^{y_1}\frac{\partial}{\partial \theta} \left(\int_0^1|\vv(\theta,y_2,t)|^2dy_2\right) d\theta\\
    &=& \int_0^1\hspace{-2mm}|\vv(0,y_2,t)|^2dy_2+2\int_0^{y_1}\hspace{-2mm}\int_0^1\hspace{-2mm}|\vv(\theta,y_2,t)|\left|\frac{\partial \vv}{\partial y_1}  (\theta,y_2,t) \right|dy_2\, d\theta\\
    &\leq& \|\vv\|^2_{L^2(S_w)}+2\int_D|\vv||\nabla \vv|dy.
    \end{eqnarray*}
Applying the H\"older inequality  and using a trace estimate in the first term  on the right-hand side of the last inequality, i.e. $\|\vv\|_{L^2(\partial D)}\leq c\|\vv\|_{L^{2}(D)}^{\frac1{2}}\|\nabla \vv\|_{L^{2}(D)}^{\frac1{2}}$, cf., e.g., \cite[Lemma 3.2]{HLN14},  we get from above
  $$ \int_0^1|\vv(y_1,y_2,t)|^2dy_2 \leq c \|\vv\|_{L^2(D)} \|\nabla \vv\|_{L^2(D)}, $$ which is valid for each $y_1\in[0,L].$ Taking the essential supremum of  $\int_0^1|\vv|^2 dy_2$ over all $0\leq y_1\leq L$ yields the statement of this lemma.
\end{proof}
From now on we use following abbreviations for norms of functions defined on  $(0,L)$: $\|\cdot\|_\infty:= \|\cdot\|_{L^\infty(0,L)}$, $ \|\cdot\|_{1, \infty}:=\|\cdot\|_{W^{1,\infty}(0,L)}$. Moreover, $C_{K, \alpha}:=C(K, \alpha, \alpha^{-1})$ denotes a constant depending on the given constants $K, \alpha$ from (\ref{assumptions}).

\subsubsection*{Time derivative  error term}

\medskip

Let us first estimate the error term  $\int_0^T \int_D \EE_1\cdot \pp^2dy dt$ on the right-hand side of  (\ref{weak_form_diff2}).  We first rewrite the components of $\EE_1$ given by (\ref{matrix_E1})  in terms of the differences  of $\barh,\, \barh_t,\, \barh_{t\, y_1}$, see also  (\ref{barh_E}),
\begin{eqnarray*}
\frac{h^1_t}{h^1}-\frac{h^2_t}{h^2}&=&\frac{\barh_t}{h^1}-\frac{h^2_t}{h^1 h^2}\barh,\\
h^2\partial_t E&=&-\barh\left(h^1_{t y_1}+\frac{h^2_t}{h^2}h^1_{y_1}\right)+\barh_{y_1}\left(h^1_{t}+\frac{h^2_t}{h^2}h^1\right)-\barh_t h^1_{y_1}+ \barh_{t y_1 }h^1.
\end{eqnarray*}
We first concentrate on those terms from $\EE_1$, which  do not involve $\barh_{ty_1}$ (see the last term of $h^2\partial_t E$).
We rewrite $\EE_1=\tilde{\EE}_1-y_2 u_1^1 \barh_{t y_1 }h^1$ and estimate the term $\int_0^T \int_D \tilde{\EE}_1\cdot \pp^2dy dt$  first.
We have 
\begin{eqnarray*}
  &&|\tilde{\EE}_1|\leq C_{K,\alpha}\Big(\underbrace{|\uu^1+\nabla \uu^1||\bar{h}_t|}_{\tilde{\EE}_1^a} 
  +\underbrace{|\uu^1+\nabla \uu^1|\big[|\bar{h}|+ |\bar{h}_{y_1}|\big]\big[|h^1_t| + |h^1_{ty_1}|+|h^2_t| \big]}_{\tilde{\EE}_1^b}\Big).
\end{eqnarray*}
The  term containing $ \tilde{\EE}^a_1$ is now estimated by Lemma~\ref{lemma:essup} applied to  $|\pp^2|$ as follows:
\begin{eqnarray*}
&&\int_0^T\int_D \tilde{\EE}^a_1\cdot \pp^2 dy dt\leq C_{K, \alpha}\int_0^T\int_0^L|\barh_t|\int_0^1|\uu^1+\nabla \uu^1||\pp^2|dy_2 dy_1 dt
  \\&& \leq C_{K, \alpha} \int_0^T \left(ess\hspace{-2mm} \sup_{0\leq y_1\leq L}\int_0^1|\pp^2|^2\right)^\frac1{2}\int_0^L|\barh_t|\left( \int_0^1|\uu^1+\nabla\uu^1|^2dy_2\right)^\frac1{2}dy_1 dt
  \\&& \leq C_{K, \alpha} \int_0^T  \|\pp^2\|_{L^{2}(D)}^\frac1{2}\|\nabla \pp^2\|_{L^{2}(D)}^\frac1{2}\|\barh_t\|_{L^2(0,L)}\|\uu^1\|_{W^{1,2}(D)}dt.
\end{eqnarray*}
Employing $ \|\pp^2\|_{L^{2}(D)}^\frac1{2}\|\nabla \pp^2\|_{L^{2}(D)}^\frac1{2}\leq c  \|\pp^2\|_{W^{1,2}(D)}$ and using the H\"older and  Young inequalities we get
\begin{eqnarray*}
  &&\int_0^T\int_D \tilde{\EE}^a_1\cdot \pp^2 dy dt
  \\ && \leq\varepsilon \int_0^T\|\pp^2\|^2_{W^{1,2}(D)}dt
  +{C}_\varepsilon C_{K,\alpha}^{2}\|\barh_t\|^2_{L^\infty(0,T;L^{2}(0,L))}\int_0^T\|\uu^1\|^{2}_{W^{1,2}(D)}dt.
\end {eqnarray*}
  Using the H\"older and the Young inequalities it follows for the second term, 
\begin{eqnarray*}
&&\int_0^T\int_D \tilde{\EE}^b_1\cdot \pp^2 dy dt
\\&& \leq C_{K, \alpha}\int_0^T\|\uu^1\|_{W^{1,2}(D)}\|\pp^2\|_{L^2(D)} \|\barh\|_{1,\infty}\big(\|h^1_t\|_{1, \infty} +\|h^2_t\|_\infty \big) dt
\\\nonumber&& 
\leq \frac{C_{K,\alpha}^2}{2}\|\barh\|_{L^\infty(0,T;W^{1,\infty}(0,L))}^2 \omega_h(T)
+ \frac{1}{2}\int_0^T\|\uu^1\|_{W^{1,2}(D)}^2\|\pp^2\|_{L^2(D)}^2dt,
\end{eqnarray*}
where $\int_0^T\|h^1_t\|_{1,\infty}^2+\|h^2_t\|_\infty^2dt=:\omega_h(T)$.\footnote{Let us point out that $\omega_h(T)\downarrow 0$ as $T\downarrow 0.$ This will be essential in Section \ref{sec:fixed_point} in order to show the contractivity.}
 Hence we can summarize that
\begin{eqnarray}\label{timeder_est2}
  &&\int_0^T\int_D \tilde{\EE}_1\cdot \pp^2 dy dt
  \\&&\leq\varepsilon \int_0^T\|\pp^2\|^2_{W^{1,2}(D)}dt+\frac{1}{2}\int_0^T\|\uu^1\|_{W^{1,2}(D)}^2\|\pp^2\|_{L^2(D)}^2dt\nonumber
  \\\nonumber&& +{C_{K,\alpha}^2}\Big({C}_\varepsilon\|\barh\|^2_{W^{1,\infty}(0,T;L^{2}(0,L))}\int_0^T\|\uu^1\|_{W^{1,2}(D)}^2dt+\|\barh\|^2_{L^\infty(0,T;W^{1,\infty}(0,L))}\omega_h(T)\Big).
\end{eqnarray}

It remains to show the estimate of the  term $-\int_0^T \int_D y_2 u_1^1 \barh_{t y_1 }h^1\psi_2^2dy dt$.
By integration by parts with respect to $y_1$ and due to the  zero boundary conditions for $\psi_2^2$ on $S_{in }\cup S_{out}$ we can rewrite this term  as
\begin{eqnarray*}
-\int_0^T \hspace{-2mm}\int_D y_2 u_1^1 \barh_{t y_1 }h^1\psi_2^2dy dt=\int_0^T \hspace{-2mm} \int_D y_2 \barh_t(h^1_{y_1}\psi_2^2u_1^1+h^1\partial_{y_1}\psi_2^2 u_1^1+h^1\psi_2^2 \partial_{y_1}u_1^1)dy dt.
\end{eqnarray*}%
Thus, by (\ref{assumptions}) we get,
\begin{eqnarray*}
  -\int_0^T \hspace{-2mm}\int_D y_2 u_1^1 \barh_{t y_1 }h^1\psi_2^2dy dt \leq C_{K, \alpha}\int_0^T \hspace{-2mm} \int_D|\barh_t||\nabla \pp^2||\uu^1|+|\barh_t|| \pp^2||\uu^1+\nabla \uu^1|dy dt.
  \end{eqnarray*}
Analogously as in term $\tilde{\EE}_1^a$ above, we apply now Lemma~\ref{lemma:essup} to $|\uu^1|$ in the first term and  to $|\pp^2|$ in the second term on the right-hand side. Consequently applying the H\"older and Young inequalities we get 
  \begin{eqnarray}\label{timeder_est_remaining}
   -\int_0^T\hspace{-2mm} \int_D y_2 u_1^1 \barh_{t y_1 }h^1\psi_2^2dy dt \leq 2C_{K, \alpha}\int_0^T\|\barh_t\|_{L^2(D)}\|\uu^1\|_{W^{1,2}(D)}\|\pp^2\|_{W^{1,2}(D)}dt\quad
    \\\leq \varepsilon \int_0^T\|\pp^2\|_{W^{1,2}(D)}^2dt+C_\varepsilon C^2_{K, \alpha}\|\barh_t\|_{L^\infty(0,T;L^2(0,L))}\int_0^T\|\uu^1\|_{W^{1,2}(D)}^2dt.\nonumber
  \end{eqnarray}%
  Taking into account  estimates (\ref{timeder_est2}), (\ref{timeder_est_remaining}),  we can conclude that 
\begin{eqnarray}\label{timeder_est3}
  &&\int_0^T\int_D {\EE}_1\cdot \pp^2 dy dt
  \\&&\leq 2\varepsilon \int_0^T\hspace{-2mm}\|\pp^2\|^2_{W^{1,2}(D)}dt+\frac{1}{2}\int_0^T\|\uu^1\|_{W^{1,2}(D)}^2\|\pp^2\|_{L^2(D)}^2dt\nonumber
  \\\nonumber&& +{C_{K,\alpha}^2}\Big(2{C}_\varepsilon\|\barh\|^2_{W^{1,\infty}(0,T;L^{2}(0,L))}\int_0^T\hspace{-2mm}\|\uu^1\|_{W^{1,2}(D)}^2dt+\|\barh\|^2_{L^\infty(0,T;W^{1,\infty}(0,L))}\omega_h(T)\Big), 
\end{eqnarray}
where  $\omega_h(T):=  \int_0^T\|h^1_t\|_{1,\infty}^2+\|h^2_t\|_\infty^2dt$ and $ C_\varepsilon=C(\varepsilon^{-1})$.

\subsubsection*{Viscous terms}
Let us start with the estimate of term containing $I_1$, cf.  (\ref{viscous_diff}). We have
\begin{eqnarray*}
\frac{\mu}{\rho}\int_0^T\int_D h^2I_1 dy dt \leq  \frac{\mu\alpha^{-1}}{\rho}\int_0^T\|\EE(\uu^1)\|_{L^2(D)}\|F_{h^2}\|_\infty\|\nabla \pp^2\|_{L^2(D)} dt.
\end{eqnarray*}
Note that by (\ref{assumptions}) one has $\|F_{h^i}\|_\infty \leq C_{K, \alpha}, \, i=1,2$.
 \\Next, we  estimate  $\EE(\uu^1)=\EE_2(\uu^1)(F_{h^1}+\EE_3)+\nabla \uu^1 \EE_3$, cf. (\ref{matrix_E}). The term $\EE_3$ given by  (\ref{matrix_E3}) can be bounded in view of (\ref{barh_E})  by 
$\|\EE_3\|_\infty \leq C_{K,\alpha}\|\barh\|_{1,\infty}$ .
Hence,
\begin{equation}\label{matrix_E_est}
\|\EE(\uu^1)\|_{L^2(D)}\leq C_{K,\alpha}\big(\|\EE_2(\uu^1)\|_{L^2(D)}(\|\barh\|_{1,\infty}+1)+\|\nabla \uu^1\|_{L^2(D)}\|\barh\|_{1,\infty}\big).
\end{equation}
In the next step  we concentrate on the estimate of $\|\EE_2(\uu^1)\|_{L^2(D)}$. Using (\ref{barh_E}) one can easily verify that $\partial_{y_1}E=\frac{1}{h^2}\Big(h^1 \barh_{y_1 y_1 }- h^1_{y_1 y_1 } \barh-\frac{h^2_{y_1}}{h^2}(h^1\barh_{y_1}-h^1_{y_1}\barh)\Big)$. Therefore, for the matrix $\EE_2(\uu^1)$  in (\ref{matrix_E2}) it holds
\begin{eqnarray}\label{matrix_E2_est}
|\EE_2(\uu^1)|&\leq& C_{K,\alpha}\left[(|\barh|+|\barh_{y_1}|)(|\uu^1|+|\nabla \uu^1|)\right.\\&& \qquad \left.+(|\barh||h^1_{y_1 y_1}|+|\barh_{y_1 y_1}||h^1|)|\uu^1|\right].\nonumber
\end{eqnarray}
Let us point out that the second derivatives $\partial^2_{y_1}\barh,\, \partial^2_{y_1}h^1$  only belong  to $L^2(0,L)$.  Therefore we use  Lemma \ref{lemma:essup} to  estimate the corresponding terms in the expression for $\|\EE_2(\uu^1)\|_{L^2(D)}^2$ in the follwing way
\begin{eqnarray*}
\int_D |\barh|^2|h^1_{y_1 y_1}|^2|\uu^1|^2& \leq& \|\barh\|^2_{\infty}\int_0^L |h^1_{y_1 y_1}|^2 dy_1\mbox{ ess\hspace{-3mm}}\sup_{0\leq y_1\leq L\ } \int_0^1|\uu^1|^2 dy_2,
\\&\leq& c_1\|\barh\|^2_{\infty}\|h^1\|_{H^2(0,L)}^2\|\nabla \uu^1\|^2_{L^{2}(D)}.
\end{eqnarray*}
Analogously,
$$\int_D |\barh_{y_1 y_1}|^2 |h^1|^2|\uu^1|^2dy\leq c_2\alpha^{-1}\|\barh\|_{H^2(0,L)}^2\|\nabla \uu^1\|^2_{L^{2}(D)}.$$
Consequently, taking into account  (\ref{matrix_E2_est}) and  estimates above we obtain
\begin{eqnarray}\label{matrix_E2_est2}
\|\EE_2(\uu^1)\|_{L^2(D)}\leq C_{K, \alpha} \big[\|\barh\|_{1,\infty}+\|\barh\|_{\infty}\|h^1\|_{H^2(0,L)}+ \|\barh\|_{H^2(0,L)}\big]\|\nabla \uu^1\|_{L^{2}(D)}.
\end{eqnarray}
Finally, using  (\ref{matrix_E_est}),(\ref{matrix_E2_est2}), Young's inequality and the fact that $\|\barh\|_{1,\infty}\leq K$, see (\ref{assumptions}), we obtain 
\begin{eqnarray}\label{viscous_diff_I1}
\frac{\mu}{\rho}\int_0^T\hspace{-2mm}\int_D h^2I_1 dy dt 
&\leq& \frac{C_{K,\alpha}\mu}{\rho}
\int_{0}^T \hspace{-2mm}\big[\|\barh\|_{1,\infty}(1+\|h^1\|_{H^2(0,L)})
\\\nonumber&&\qquad \qquad + 
\|\barh\|_{H^2(0,L)}\big]
\|\uu^1\|_{W^{1,2}(D)}
\|\nabla\pp^2\|_{L^2(D)}dt 
\\\nonumber &\leq& \varepsilon\int_0^T\|\nabla\pp^2\|_{L^2(D)}^2dt+C_\varepsilon C_{K,\alpha}^2\|\uu^1 \|_{L^2(0,T;W^{1,2}(D))}^2
\\&&  \qquad \qquad.\Big(\|\barh \|^2_{L^\infty(0,T;W^{1,\infty}(0,L))}C_{h}^2+\|\barh \|^2_{L^\infty(0,T;H^{2}(0,L))}\Big),
\nonumber
\end{eqnarray}
where
 $C_\varepsilon =C\big(\varepsilon^{-1},\frac{\mu}{\rho}\big)$ and
\begin{equation}\label{C_h}
C_h:=  1+ \|h^1 \|_{L^\infty(0,T;H^{2}(0,L))}+\|h^2 \|_{L^\infty(0,T;H^{2}(0,L))}\
\end{equation}
is bounded.

\medskip

Now we proceed with the estimate of term containing $I_2$ in  (\ref{viscous_diff}),
\begin{eqnarray*}
\frac{\mu}{\rho}\int_0^T\int_D h^2I_2 dy dt \leq  \frac{\mu\alpha^{-1}}{\rho}\int_0^T\|F_{h^1}\|_\infty\|\nabla \uu^1\|_{L^2(D)}\|\EE(\pp^1)\|_{L^2(D)} dt.
\end{eqnarray*} 
We recall that $\EE(\pp^1)=\EE_2(\pp^1)(F_{h^1}+\EE_3)+\nabla \pp^1 \EE_3$. Analogously as in (\ref{matrix_E_est}) we can bound
\begin{equation}\label{matrix_E_est_pp}
\|\EE(\pp^1)\|_{L^2(D)}\leq C_{K,\alpha}\big(\|\EE_2(\pp^1)\|_{L^2(D)}(\|\barh\|_{1,\infty}+1)+\|\nabla \pp^1\|_{L^2(D)}\|\barh\|_{1,\infty}\big).
\end{equation}
Repeating similar estimates as in (\ref{matrix_E2_est}), (\ref{matrix_E2_est2}) we obtain
\begin{eqnarray}\label{matrix_E2_est2_pp}
\|\EE_2(\pp^1)\|_{L^2(D)}\leq C_{K, \alpha} \big[\|\barh\|_{1,\infty}(1+\|h^1\|_{H^2(0,L)})+ \|\barh\|_{H^2(0,L)}\big]\|\nabla \pp^1\|_{L^{2}(D)}.
\end{eqnarray}
Consequently, using (\ref{matrix_E_est_pp}),  (\ref{matrix_E2_est2_pp}) we arrive at  
\begin{eqnarray}\label{viscous_diff_I2}
\frac{\mu}{\rho}\int_0^T\hspace{-2mm}\int_D h^2I_2 dy dt 
&\leq& \frac{C_{K,\alpha}\mu}{\rho}
\int_{0}^T \hspace{-2mm}\big[\|\barh\|_{1,\infty}(1+\|h^1\|_{H^2(0,L)})
\\\nonumber&&\qquad \qquad + 
\|\barh\|_{H^2(0,L)}\big]
\|\uu^1\|_{W^{1,2}(D)}
\|\nabla\pp^1\|_{L^2(D)}dt. 
\end{eqnarray}
In order to replace the norm of $\nabla \pp^1$ by the norm of  $\nabla \pp^2$ we now use the relation  $\pp^1=\RR^{-1}\pp^2$ cf. (\ref{pps}). 
One can verify that 
\begin{eqnarray*}
|\nabla (\RR^{-1}\pp^2) |&\leq& C_{K, \alpha}|\nabla \pp^2|+\big|\partial_{y_1}(\hfraci)(1-y_2 h^1_{y_1}-h^1)-y_2h^1\partial_{y_1}^2(\hfraci)\big||\psi^2_1|
\\& \leq& C_{K, \alpha}\Big(|\nabla \pp^2|+| \pp^2|(1+|\partial_{y_1}^2 h^1|+|\partial_{y_1}^2h^2|)\Big),\quad \mbox{thus}
\end{eqnarray*} 
\begin{equation}\label{norm_pp1}
\|\nabla \pp^1 \|_{L^2(D)}\leq C_{K, \alpha}\Big(\| \pp^2 \|_{W^{1,2}(D)}
(1+ \|h^1\|_{H^2(0,L)}+ \|h^2\|_{H^2(0,L)})\Big).
\end{equation} 
Inserting (\ref{norm_pp1})  into (\ref{viscous_diff_I2}) we finally  obtain an analogous estimate to (\ref{viscous_diff_I1}), i.e.,
\begin{eqnarray}\label{viscous_diff_I2final}
\frac{\mu}{\rho}\int_0^T\hspace{-2mm}\int_D h^2I_2 dy dt 
&\leq &\varepsilon \int_0^T \|\pp^2\|_{W^{1,2}(D)}^2dt+\bar{C}_\varepsilon C_{K,\alpha}^2C_h^2\|\uu^1 \|_{L^2(0,T;W^{1,2}(D))}^2\qquad
\\&&\qquad \qquad . \Big(\|\barh \|^2_{L^\infty(0,T;W^{1,\infty}(0,L))}C_{h}^2+\|\barh \|^2_{L^\infty(0,T;H^{2}(0,L))}\Big)\nonumber
\end{eqnarray}
with $\bar{C}_\varepsilon=\bar{C}\big(\varepsilon^{-1},\frac{\mu}{\rho}\big)$. 

\medskip
{\rred
It remains to show the estimate of  the term containing $I_3$, cf.  (\ref{viscous_diff}). Since
\begin{eqnarray*}%
  \EE_\RR=\begin{bmatrix}
  \frac{\barh}{h^2}  & 0\\-y_2E &0 \end{bmatrix}
  \ \mbox{and}\quad \RR^T\EE_{\RR} +\EE_{\RR}^T= \begin{bmatrix}
 \frac{\barh}{h^2}\frac{h^2+h^1}{h^2}+y_2^2E^2  & -y_2E\\-y_2E &0 \end{bmatrix},
\end{eqnarray*}%
with (\ref{barh_E}) we have $\|\RR^T\EE_{\RR} +\EE_{\RR}^T\|_\infty\leq C_{K,\alpha}\|\barh\|_{1,\infty}$. Thus 
\begin{eqnarray*}
\frac{\mu}{\rho}\int_0^T\int_D h^2I_3 dy dt \leq  \frac{\mu}{\rho}\int_0^TC_{K,\alpha}\|\barh\|_{1,\infty} \|F_{h^1}\|_\infty^2\|\nabla \uu^1\|_{L^2(D)}\|\nabla \pp^1\|_{L^2(D)} dt.
\end{eqnarray*}
Using (\ref{norm_pp1}), (\ref{C_h}), and by applying  Young's inequality, we get 
\begin{eqnarray}\label{viscous_diff_I3}
&&\frac{\mu}{\rho}\int_0^T\int_D h^2 I_3 dy dt \leq \frac{\mu}{\rho}\int_0^TC_{K,\alpha}^4C_h \|\barh\|_{1,\infty} \|\nabla \uu^1\|_{L^2(D)}\| \pp^2\|_{W^{1,2}(D)} 
dt\ 
\\&&\leq \varepsilon \int_0^T\| \pp^2\|_{W^{1,2}(D)}^2dt+\tilde{C}_\varepsilon C_{K,\alpha}^8C_h^2\|\uu^1 \|_{L^2(0,T;W^{1,2}(D))}^2\|\barh \|^2_{L^\infty(0,T;W^{1,\infty} (0,L))}\nonumber,
\end{eqnarray}
where
$\tilde{C}_\varepsilon=\bar{C}\big(\varepsilon^{-1},\frac{\mu}{\rho}\big)$.
}
Summarizing the above estimates of viscous error terms
 (\ref{viscous_diff_I1}), (\ref{viscous_diff_I2final}), (\ref{viscous_diff_I3}) we finally get 
\begin{eqnarray}\label{viscous_diff_Is}
&&\frac{\mu}{\rho}\int_0^T\int_D h^2I_1+h^2I_2-h^2I_3 dy dt
\\&&\leq3\varepsilon  \int_0^T\|\pp^2\|_{W^{1,2}(D)}dt +C_{\varepsilon}(C_{K,\alpha}^2+C_{K,\alpha}^8)C_h^2\nonumber
\\&&\qquad \qquad 
.\Big(\|\barh \|^2_{L^\infty(0,T;W^{1,\infty}(0,L))}+\|\barh \|^2_{L^\infty(0,T;H^{2}(0,L))}\Big)\|\uu^1 \|_{L^2(0,T;W^{1,2}(D))}^2,\nonumber
\end{eqnarray}
where 
$C_{\varepsilon}=C\big( \varepsilon^{-1}, \frac{\mu}{\rho})$ and $C_h$ is given by (\ref{C_h}).

\subsubsection*{Convective terms}
{\rred
We follow with the estimate  terms coming from the difference of convective terms, 
see  (\ref{convective_diff}).

{\bf I)}  Recalling (\ref{becko-rozdiel}),  (\ref{assumptions}) and  using $\|\RR\|_{\infty}\leq C_{K,\alpha}$, cf. (\ref{matrix_R}), for the first term on the right-hand side of (\ref{convective_diff}) it holds
\begin{eqnarray*}
\int_0^T| b_{h^2}(\pp^2,\RR\uu^1,\pp^2)| dt\leq \frac{1}{2}\int_0^T| B_{h^2}(\pp^2,\RR\uu^1,\pp^2)|+| B_{h^2}(\pp^2,\pp^2,\RR\uu^1)|dt\\
\leq C_{K,\alpha}\int_0^T\int_0^D |\pp^2|^2|\nabla \uu^1| +|\pp^2||\nabla \pp^2||\uu^1|dy\, dt.
\nonumber
\end{eqnarray*}
Further, applying the H\"older inequality, the Sobolev interpolation inequality in two dimensions: $\|\varphi\|_{L^4(D)}\leq c\|\nabla \varphi\|_{L^2(D)}^{1 / 2}\|\varphi\|_{L^2(D)}^{1/2}$,
and  Young's inequality
\begin{equation}\label{young_e}
  ab \leq\varepsilon a^p+C_\varepsilon b^{\frac{p}{p-1}},\quad  C_\varepsilon=C(\varepsilon^{-1}), \  1<p<\infty
  \end{equation}
 for  $p=2,4$ we finally get
\begin{eqnarray}\label{b1_estim}
\int_0^T| b_{h^2}(\pp^2,\RR\uu^1,\pp^2)| dt&\leq& 
 C_{K,\alpha}\int_0^T\|\pp^2\|_{L^2}\|\nabla \pp^2\|_{L^2}\|\nabla \uu^1\|_{L^2}
\\&&+\|\pp^2\|^{1/2}_{L^2}\|\nabla \pp^2\|^{3/2}_{L^2}\| \uu^1\|^{1/2}_{L^2}\| \nabla \uu^1\|^{1/2}_{L^2}dt\nonumber\\
&& \hspace{-2cm}\leq 2\varepsilon \int_0^T\|\nabla \pp\|_{L^{2}}^2dt+C^{\RM{1}}_\varepsilon C_{K,\alpha}^2
\int_0^T \|\nabla \uu^1\|_{L^2}^2\|\pp^2\|_{L^2}^2 dt.\nonumber
\end{eqnarray}
Here we have used the abbreviation $\|\cdot\|_{L^2}:=\|\cdot\|_{L^2(D)}$. 
Note that ${C^{\RM{1}}_\varepsilon}$ is a constant depending on 
${\varepsilon}^{-1}$ and on the norm $\|\uu^1\|_{L^\infty(0,T;L^2(D))}^2$.

\medskip

{\bf II)} With the assistance of (\ref{becko-rozdiel}), (\ref{assumptions}) and  recalling $\|\RR\|_{\infty}\leq C_{K,\alpha}$ 
  we can write for the second term on the right-hand side of (\ref{convective_diff})
\begin{eqnarray}\label{b2_estim0}
\int_0^T\hspace{-2mm}\left|\hfracdi b_{h^2}(\RR\uu^1,\RR\uu^1,\pp^2)\right|dt \leq C_{K,\alpha}\hspace{-2mm}\int_0^T \hspace{-2mm}\|\barh \|_{\infty}\hspace{-2mm}\int_D\hspace{-1mm}|\uu^1||\nabla \uu^1||\pp^2|+|\uu^1|^2|\nabla \pp^2|dy dt.\ 
\end{eqnarray}
Now, using  H\"older and  Sobolev inequalities as above we obtain
\begin{eqnarray}\label{b2_estim}
 \hspace{-7mm}&& \int_0^T\left|\hfracdi b_{h^2}(\RR\uu^1,\RR\uu^1,\pp^2)\right|dt \leq 
  \\\hspace{-7mm}&&  C_{K,\alpha}\hspace{-1mm}\int_0^T \hspace{-3mm}\|\barh \|_{\infty} \hspace{-0mm}\Big (
\|\nabla \uu^1\|_{L^2}^{\frac{3}{2}}\| \uu^1\|_{L^2}^{\frac1{2}}\| \pp^2\|_{L^2}^{\frac1{2}}\| \nabla\pp^2\|_{L^2}^{\frac1{2}}
 + \|\uu^1\|_{L^2}\|\nabla \uu^1\|_{L^2}\| \nabla\pp^2\|_{L^2} \Big)dt.\nonumber
\end{eqnarray}
With additional use of  Young's inequality
 $ ab \leq\frac1{p} a^p+{\frac{p-1}{p}}\, b^{\frac{p}{p-1}}, \  1<p<\infty$
with $p=4$  the first term on the right-hand side of (\ref{b2_estim}) can be  bounded from above by
\begin{eqnarray*}
&&\frac{1}{4}\int_0^T \| \pp^2\|_{L^2}^{2} \|\nabla \uu^1\|_{L^2}^{2}dt
\\&+&\frac{3}{4}C^{4/3}_{K,\alpha}\|\barh\|^{4/3}_{L^\infty((0,T)\times(0,L))}\|\uu^1\|^{2/3}_{L^\infty(0,T;L^{2}(D))}\int_0^T \|\nabla \pp^2\|_{L^2}^{2/3} \|\nabla \uu^1\|_{L^2}^{4/3}dt.
\end{eqnarray*}
Consequently, by applying Young's inequality (\ref{young_e}) with $p=3$
we get for this  term 
\begin{eqnarray*}
&&\frac{1}{4}\int_0^T \| \pp^2\|_{L^2}^{2} \|\nabla \uu^1\|_{L^2}^{2}dt+\varepsilon \int_0^T \|\nabla \pp^2\|_{L^2}^{2}dt
\\&+&\frac{3}{4}C_\varepsilon C^{2}_{K,\alpha}\|\barh\|^{2}_{L^\infty((0,T)\times(0,L))}\|\uu^1\|_{L^\infty(0,T;(L^{2}(D))}\int_0^T \|\nabla \uu^1\|_{L^2}^{2}dt.
\end{eqnarray*}
The second right-hand side term in (\ref{b2_estim}) can be  bounded using Young's inequality (\ref{young_e}) with $p=2$   by
$$ \varepsilon \int_0^T \|\nabla \pp^2\|_{L^2}^{2}dt+C_\varepsilon C^{2}_{K,\alpha}\|\barh\|^{2}_{L^\infty((0,T\times(0,L))}\|\uu^1\|^2_{L^\infty(0,T;L^{2}(D))}\int_0^T \|\nabla \uu^1\|_{L^2}^{2}dt.$$
We finally  have
\begin{eqnarray}\label{b2_estim2}
&&\int_0^T\hspace{-2mm}\left|\hfracdi b_{h^2}(\RR\uu^1,\RR\uu^1,\pp^2)\right|dt  \leq \frac{1}{4}\int_0^T\hspace{-2mm} \| \pp^2\|_{L^2}^{2} \|\nabla \uu^1\|_{L^2}^{2}dt
\\&&\ \nonumber+2\varepsilon \int_0^T \hspace{-2mm}\|\nabla \pp^2\|_{L^2}^{2}dt
 +C^{\RM{2}}_\varepsilon C^{2}_{K,\alpha}\|\barh\|^{2}_{L^\infty((0,T)\times(0,L))}\int_0^T \|\nabla \uu^1\|_{L^2}^{2} dt,
\end{eqnarray}
where  ${C^{\RM{2}}_\varepsilon}$ is a constant depending on 
${\varepsilon}^{-1}$ and on the norm $\|\uu^1\|_{L^\infty(0,T;L^2(D))}$.

\medskip

{\bf III)} Now we estimate  the third term on the right-hand side of  (\ref{convective_diff}). Let us remark that with the assistance of  (\ref{barh_E}) one can easily show that
$|(\nabla \RR)_{i,j}|\leq C_{K,\alpha}(|\barh|(1+|h^1_{y_1 y_1}|)+|\barh_{y_1}|+|\barh_{y_1 y_1}|)$, compare  e.g. the lines above (\ref{matrix_E2_est}). Thus
\begin{eqnarray*} \label{b3_estim0}
  \int_0^T\hspace{-2mm} \left|\hfraci T_{h^2}(\RR\uu^1,\uu^1,\pp^2)\right|dt\leq
 C_{K,\alpha} \int_0^T \hspace{-2mm}\int_D |\uu^1|^2|\pp^2|(\|\barh\|_{1,\infty}(1+|h^1_{y_1 y_1}|)+|\barh_{y_1 y_1}|).
\end{eqnarray*}
We firstly apply the H\"older inequality  with respect to  the integral $\int_0^1 dy_2$ on the right-hand side,
\begin{eqnarray*}
  &&\hspace{-5mm}\int_0^T\hspace{-2mm} \left|\hfraci T_{h^2}(\RR\uu^1,\uu^1,\pp^2)\right|dt\leq
  \\&&\hspace{-5mm} C_{K,\alpha} \hspace{-1mm}\int_0^T\hspace{-2mm} \int_0^L\hspace{-2mm}\big(\|\barh\|_{1,\infty}(1+|h^1_{y_1 y_1}|)+|\barh_{y_1 y_1}|\big)
\|\uu^1\|_{L^2(0,1)}\|\uu^1\|_{L^4(0,1)}\|\pp^2\|_{L^4(0,1)} dy_1 dt,
\end{eqnarray*} 
next we  apply  Lemma \ref{lemma:essup} to $\|\uu^1\|_{L^2(0,1)}$ and the H\"older inequality with respect to the integral $\int_0^L dy_1$, 
\begin{eqnarray*}
  &&\hspace{-5mm}\int_0^T\hspace{-2mm} \left|\hfraci T_{h^2}(\RR\uu^1,\uu^1,\pp^2)\right|dt\leq\\
  &&\hspace{-5mm}  C_{K,\alpha} \int_0^T\hspace{-1mm}\big(\|\barh\|_{1,\infty}(1+\|h^1\|_{H^2(0,L)})+\|\barh\|_{H^2(0,L)}\big)\|\nabla \uu^1\|^{\frac1{2}}_{L^2} \| \uu^1\|^{\frac1{2}}_{L^2} 
\|\uu^1\|_{L^4}\|\pp^2\|_{L^4}, 
\end{eqnarray*} 
where $\|\cdot\|_{L^r}:=\|\cdot\|_{L^r(D)},\, r=2,4$.
Finally, using  the {\red Sobolev interpolations as in the first item I)},  Young's inequalities
with $p=4$ and (\ref{young_e}) with $p=3$, 
we obtain 
\begin{eqnarray}\label{b3_estimB}
  &&\hspace{-5mm}\int_0^T\hspace{-2mm} \left|\hfraci T_{h^2}(\RR\uu^1,\uu^1,\pp^2)\right|dt\\
  &&\nonumber\leq C_{K,\alpha}\| \uu^1\|_{L^\infty(0,T;L^2(D))}
\\\nonumber&&\quad \big(\|\barh\|_{L^\infty(0,T; W^{1,\infty}(0,L))} (1+\|h^1\|_{L^\infty(0,T; H^2(0,L))})+\|\barh\|_{L^\infty(0,T; H^2(0,L))}\big).
\\ \nonumber&&\quad \int_0^T\hspace{-1mm}
\|\nabla \uu^1\|_{L^2(D)}\|\pp^2\|^{1/2}_{L^2(D)}\|\nabla \pp^2\|^{1/2}_{L^2(D)}dt
\\\nonumber && \leq\frac{1}{4}\int_0^T\hspace{-2mm} \| \pp^2\|_{L^2(D)}^{2} \|\nabla \uu^1\|_{L^2(D)}^{2}dt
+\varepsilon \int_0^T \hspace{-2mm}\|\nabla \pp^2\|_{L^2(D)}^{2}dt+
\\\nonumber&& \quad {C}^{\RM{3}}_\varepsilon C^{2}_{K,\alpha}
\big(C_{h}^2\|\barh\|_{L^\infty(0,T; W^{1,\infty}(0,L))}^2+\|\barh\|^2_{L^\infty(0,T; H^2(0,L))}\big)
\int_0^T \hspace{-2mm}\|\nabla \uu^1\|_{L^2(D)} dt,
\end{eqnarray} 
where   ${{C}^{\RM{3}}_\varepsilon}$ is a constant depending on 
${\varepsilon}^{-1}$ and on the norm of $\uu^1$ in ${L^\infty(0,T;L^2(D))}$.  We  recall  $1+\|h^1\|^2_{L^\infty(0,T; H^2(0,L))} \leq C_h$ given by (\ref{C_h}).

Thus, since $C_h\geq 1$ 
we can summarize the estimate of the third term on the right-hand side of  (\ref{convective_diff}), 
\begin{eqnarray}\label{b3_estim2}
\int_0^T\hspace{-2mm}|\hfraci T_{h^2}(\RR\uu^1,\uu^1,\pp^2)|dt \leq \frac{1}{4}\int_0^T\hspace{-2mm} \| \pp^2\|_{L^2}^{2} \|\nabla \uu^1\|_{L^2}^{2}dt
+\varepsilon \int_0^T \hspace{-2mm}\|\nabla \pp^2\|_{L^2}^{2}dt \qquad 
\\\nonumber +C^{\RM{3}}_\varepsilon C^{2}_{K,\alpha}C_h^2
\big[\|\barh\|_{L^\infty(0,T; W^{1,\infty}(0,L))}^2+\|\barh\|^2_{L^\infty(0,T; H^2(0,L))}\big]\int_0^T\|\nabla \uu^1\|_{L^2(D)}dt.
\end{eqnarray}

 \medskip

 {\bf IV)} To the estimate the fourth term on the right-hand side of (\ref{convective_diff}) we proceed analogously as in item {\bf III)} to come to
\begin{eqnarray}\label{b4_estimB}
 && \int_0^T\hspace{-2mm} \left|\hfraci T_{h^2}(\RR\uu^1,\pp^1,\RR\uu^1)\right|dt
 \leq C_{K,\alpha}\| \uu^1\|_{L^\infty(0,T;L^2(D))}.
\\\nonumber&&\quad \big(\|\barh\|_{L^\infty(0,T; W^{1,\infty}(0,L))} (1+\|h^1\|_{L^\infty(0,T; H^2(0,L))})+\|\barh\|_{L^\infty(0,T; H^2(0,L))}\big).
\\ \nonumber&&\quad \int_0^T\hspace{-1mm}
\|\nabla \uu^1\|_{L^2(D)}\|\pp^1\|^{1/2}_{L^2(D)}\|\nabla \pp^1\|^{1/2}_{L^2(D)}dt,
\end{eqnarray}
cf. (\ref{b3_estimB}). Now replacing the norms of $\pp^1$ by norms of  $\pp^2$ 
cf. (\ref{pps}) and (\ref{norm_pp1}) we get constant $ C^2_{K,\alpha}$ on the right-hand side of (\ref{b4_estimB}), however we for the sake of simplicity we denote it again as $ C_{K,\alpha}$. Moreover additional factor  $ (1+\|h^1\|_{L^\infty(0,T; H^2(0,L))}+\|h^2\|_{L^\infty(0,T; H^2(0,L))})^{1/2}$ bounded by $C_h^{\frac1{2}}$, cf. (\ref{C_h}), appears. Thus using Young's inequalities as above we finally come to
\begin{eqnarray}\label{b4_estim2}
\int_0^T\hspace{-2mm}|\hfraci T_{h^2}(\RR\uu^1,\pp^1,\RR\uu^1)|dt \leq \frac{1}{4}\int_0^T\hspace{-2mm} \| \pp^2\|_{L^2}^{2} \|\nabla \uu^1\|_{L^2}^{2}dt
+\varepsilon \int_0^T \hspace{-2mm}\|\pp^2\|_{W^{1,2}}^{2}dt \qquad 
\\\nonumber +C^{\RM{4}}_\varepsilon C^{2}_{K,\alpha}C_h^3
\big[\|\barh\|_{L^\infty(0,T; W^{1,\infty}(0,L))}^2+\|\barh\|^2_{L^\infty(0,T; H^2(0,L))}\big]\int_0^T\|\nabla \uu^1\|_{L^2(D)}dt.
\end{eqnarray}
where   ${{C}^{\RM{4}}_\varepsilon}$ is a constant depending on 
${\varepsilon}^{-1}$ and on the norm of $\uu^1$ in ${L^\infty(0,T;L^2(D))}$,  $C_h\geq 1$ is given by (\ref{C_h}).

Summarizing the above estimates of convective error terms
(\ref{b1_estim}), (\ref{b2_estim2}), (\ref{b3_estim2}) and  (\ref{b4_estim2}) we finally get for the difference of convective terms  (\ref{convective_diff})

\begin{eqnarray}\label{convective_diff_estimate}
&&\int_0^T b_{h^2}(\pp^2,\RR\uu^1,\pp^2)+ \hfracdi b_{h^2}(\RR\uu^1,\RR\uu^1,\pp^2)
\\&& \hphantom{\int_0^T} -\frac1{2}\int_D \hfraci[ T_{h^2}(\RR\uu^1,\uu^1,\pp^2)-T_{h^2}(\RR\uu^1,\pp^1,\RR\uu^1)]dy dt.\nonumber
\\\nonumber&&\leq\left(\frac{3}{4}+C_\veps^{\RM{1}}C_{K,\alpha}^2\right)\int_0^T\hspace{-2mm} \| \pp^2\|_{L^2}^{2} \|\nabla \uu^1\|_{L^2}^{2}dt
+6\varepsilon \int_0^T \hspace{-2mm}\|\pp^2\|_{W^{1,2}}^{2}dt \qquad 
\\&&\nonumber +C^{2}_{K,\alpha}(C_\veps^{\RM{2}}+C_\veps^{\RM{3}}C_h^2 +C^{\RM{4}}_\varepsilon C_h^3)
\big[\|\barh\|_{L^\infty(0,T; W^{1,\infty}(0,L))}^2+\|\barh\|^2_{L^\infty(0,T; H^2(0,L))}\big].
\\&&\nonumber\qquad \int_0^T\|\nabla \uu^1\|_{L^2(D)} +\|\nabla \uu^1\|^2_{L^2(D)}dt.
\end{eqnarray}
}

\subsubsection*{Boundary terms}
It remains to estimate the boundary terms in (\ref{weak_form_diff2}), cf. (\ref{bnd1}), (\ref{bnd2}), (\ref{bnd3}). This can be done with the use of the following {trace inequality for $\vv \in \VV_{\hspace{-1mm}div}$,
  $$\|\vv\|_{L^r(\partial D)}\leq c\|\vv\|_{L^{2}(D)}^{\frac1{r}}\|\nabla \vv\|_{L^{2}(D)}^{1-\frac1{r}},\, 0<r<\infty,$$ cf. \cite[Lemma 3.2]{HLN14}}. 
As first we demonstrate the estimate of pressure  terms containing  $q_{in}$.
The pressure terms on the outflow $S_{out}$ and the bottom boundary $S_w$
can be bounded analogously. We denote  $q_{in}^1- q_{in}^2=:\bar{q}_{in}$. We have
\begin{eqnarray*}
  &&\int_0^T\int_0^1|q_{in}^1- q_{in}^2||h^1| \psi^2_1(0,y_2,t)|dy_2 dt
  \leq \int_0^T \alpha^{-1}\|\bar{q}_{in}\|_{L^2(S_{in})}\|\pp^2\|_{W^{1,2}(D)}dt
\\&& \leq \varepsilon \int_0^T\|\pp^2\|_{W^{1,2}(D)}^2dt + C_\varepsilon \alpha^{-2}\int_0^T\|\bar{q}_{in}\|^2_{L^2(S_{in})} dt.
\end{eqnarray*}
The remaining  boundary terms on $S_w$ containing $h^1-h^2$ and $h^1_t- h^2_t$  can be estimated  as follows. By the H\"older inequality we have
\begin{eqnarray*}
  \int_0^T\hspace{-2mm}\int_0^L\hspace{-1mm}\frac1{2} \left|\frac{h^2}{h^1}\right||h^1_t- h^2_t|| u^1_2 \psi^2_2|(y_1,1,t)+ \left|\frac{h^2-h^1}{h^1}\right|\left|q_w^1+\frac{h^2_t}{2}u^1_2\right||\psi^2_2(y_1,1,t)|dy_1 dt
  \\\leq \int_0^T\hspace{-2mm}\frac{\alpha^{-2}}{2}\|\barh_t\|_{L^2(0,L)}\|\uu^1\|_{L^4(S_w)}\|\pp^2\|_{L^4(S_w)} 
  +\alpha^{-1}\|q_w^1\|_{L^2(0,L)}\|\barh\|_{L^\infty(0,L)}\|\pp^2\|_{L^2(S_w)} 
   \\+ \frac{\alpha^{-1}}{2}\|\barh\|_{L^\infty(0,L)}\|h^2_t\|_{L^2(0,L)}\|\uu^1\|_{L^4(S_w)}\|\pp^2\|_{L^4(S_w)} dt
\end{eqnarray*}
Using the trace argument mentioned above,  $\|\vv\|_{L^r(S_w)}\leq c \|\nabla \vv\|_{L^{2}(D)}, \, r=2,4$ and 
applying the Young inequality (\ref{young_e}) with $p=2$ we obtain for the remaining boundary terms on $S_w$
\begin{eqnarray*}
   \leq 3\varepsilon \int_0^T\hspace{-2mm}\|\pp^2\|_{W^{1,2}(D)}^2dt
  +{C}_\varepsilon C_{K,\alpha}^{2} \hspace{-0mm}\|\barh_t\|_{L^\infty(0,T;L^2(0,L))}^2
  \int_0^T\hspace{-2mm}\|\nabla \uu^1\|_{L^{2}(D)}^2dt,
  \\+{C}_\varepsilon C_{K,\alpha}^{2}\|\barh\|_{L^\infty(0,T;L^\infty(0,L))}^2(1+c_h^2) \int_0^T\hspace{-2mm}\|\nabla \uu^1\|_{L^{2}(D)}^2+\|q_w^1\|^2_{L^2(0,L)}dt.
\end{eqnarray*}
where $c_{h}:=\|h^2_t\|_{L^\infty(0,T;L^2(0,L))}$. 

Thus, collecting all estimates of the boundary terms we obtain
\begin{eqnarray}\label{boundary_estim}
&&\int_0^T\int_0^1-(q_{in}^1- q_{in}^2)h^1 \psi^2_1(0,y_2,t)
+(q_{out}^1- q_{out}^2)h^1 \psi^2_1(L,y_2,t)dy_2  dt
\\&&-\int_0^T\int_0^L(q_{w}^1- q_{w}^2)\psi^2_2(y_1,1,t) + \frac{h^2}{2h^1} \partial_t (h^1- h^2)u^1_2\psi^2_2(y_1,1,t) \nonumber
\\&&\hphantom{\int_0^T\int_0^1} +\hfracdi \big(q_w^1+ \frac{h^2_t}{2}u^1_2\big)\psi_2^2(y_1,1,t)dy_1 dt\nonumber
\\&&\leq 6\varepsilon \int_0^T\hspace{-2mm}\|\pp^2\|_{W^{1,2}(D)}^2dt
+ \tilde{C}_\varepsilon C_{K,\alpha}^2\int_0^T\hspace{-2mm}\|\bar{q}_{in}\|^2_{L^2(S_{in})}\hspace{-1mm} +\|\bar{q}_{out}\|^2_{L^2(S_{out})}\hspace{-1mm}+ \|\bar{q}_{w}\|^2_{L^2(S_{w})} dt\nonumber
\\&& \ +\tilde{C}_\varepsilon C_{K,\alpha}^2\tilde{c}_{h}( \| \barh\|_{W^{1,\infty}(0,T;L^2(0,L))}^2 +\|\barh\|_{L^\infty(Q)}^2)\int_0^T\hspace{-2mm}\|\nabla \uu^1\|^2_{L^2(D)}+\|q_w^1\|^2_{L^2(0,L)}dt, \nonumber
\nonumber 
\end{eqnarray}
where $Q:=(0,T)\times(0,L)$,  $\tilde{C}_\varepsilon$ 
depends on $\varepsilon^{-1}$ and  $\tilde{c}_{h}:=1+\|h^2_t\|_{L^\infty(0,T;L^2(0,L))}^2$ is bounded. 

\subsection{Final estimate}

\label{sec:final_estim}

Let us summarize the above   estimates of the right-hand side of (\ref{weak_form_diff2}).
Note that the constants $C^{\RM{1}-\RM{4}}_\varepsilon$,  $C_h$  and $\tilde{c}_h$ in the estimates 
(\ref{timeder_est3}),  (\ref{viscous_diff_Is}), 
(\ref{convective_diff_estimate}),   (\ref{boundary_estim}) depend on the positive powers of the norms $\|\uu^i\|_{L^\infty(0,T;L^2(D))},  \, \|h^i\|_{L^\infty(0,T;H^2(0,L))}, \, i=1,2$, cf. (\ref{C_h}) and $\|h^2\|_{W^{1,\infty}(0,T;L^2(0,L))}$, respectively. 
\\Applying the Korn inequality we finally get from (\ref{weak_form_diff2}): 
\begin{eqnarray}\label{weak_form_diff3}
&&g(T)+ \frac{\alpha \mu }{\rho}\tilde{c}_{Ko}\int_0^T\|\pp^2\|_{W^{1,2}(D)}^2+\frac{c}{E}\|\partial^2_{y_1}\xi\|_{L^2(0,L)}^2 dt \\ \nonumber && \qquad 
 \leq{\rred 17} \varepsilon\int_0^T\|\pp^2\|_{W^{1,2}(D)}^2+C_1
\int_0^T\|\uu^1\|_{W^{1,2}(D)}^2\|\pp^2\|_{L^2(D)}^2 dt+\vartheta(T), \\&& \mbox{where}\nonumber
\\\nonumber &&
g(t):=\frac{\alpha}{2}\|\pp^2\|_{L^2(D)}^2(t)+\frac{1}{2E}\|\xi\|_{L^2(0,L)}^2(t)\\&&\nonumber \qquad \qquad+\frac{bE}{2}\|\bar{\eta}\|_{L^2(0,L)}^2(t)
+\frac{aE}{2}\|\partial_{y_1}\bar{\eta}\|_{L^2(0,L)}^2(t),
\\\nonumber
&&\vartheta(t):=C_2\int_0^t\|\bar{q}_{in}\|_{L^2(S_{in})}^2+\|\bar{q}_{out}\|_{L^2(S_{out})}^2+\|\bar{q}_{w}\|_{L^2(S_w)}^2ds\\\nonumber
&&\hphantom{....}+C_3\omega(t)\left[ \|\barh\|^2_{L^\infty(0,t;W^{1,\infty}(0,L))}+ \|\barh\|^2_{L^\infty(0,t;H^{2}(0,L))}
 +\|\barh\|_{W^{1,\infty}(0,T;L^2(0,L))}^2 \right]
,
\\\nonumber &&
\omega(t):=\int_0^t\hspace{-1mm}\|\uu^{1}\|_{W^{1,2}(D)}+\|\uu^{1}\|^2_{W^{1,2}(D)}+\|h^1_t\|_{1,\infty}^2+\|h^2_t\|_\infty^2
+{\rred 
  \|q^1_{w}\|_{L^2{(S_{w})}}^2 }
ds, 
\end{eqnarray}
 $\omega(t)\downarrow 0$ as $t\downarrow 0$, and 
\\ $C_1=2+C^\RM{1}_\varepsilon C_{K,\alpha}^2$, 
$C_2=C_\varepsilon C_{K,\alpha}^2$,
\\$C_3=C^2_{K,\alpha}[{\rred 1+C_\varepsilon}+C_\varepsilon C_h^2(1+C^6_{K,\alpha}) +C_\varepsilon^\RM{2}+C_h^2C_\varepsilon^\RM{3}+ C_h^3C_\varepsilon^\RM{4}+C_\varepsilon\tilde{c}_h]$, \\here
$C^\varepsilon$ denotes a constant depending on $\varepsilon^{-1}$ and $C_{K,\alpha}$.

Now, we chose  an $\varepsilon$ small enough such that $\frac{\alpha \mu }{\rho}\tilde{c}_{Ko}-{\rred 17} \varepsilon\geq \frac{\alpha \mu }{2\rho}\tilde{c}_{Ko}$, i.e. $\varepsilon\leq\frac{\alpha \mu }{\rred 34\rho}\tilde{c}_{Ko}$.
Thus, the constants of type $C_\varepsilon$ (and consequently the constants $C_1,C_1,C_3$) depend on $(\frac{\alpha \mu }{\rho}\tilde{c}_{Ko})^{-1}$. By this choice of $\varepsilon$, the inequality (\ref{weak_form_diff3}) yields
\begin{eqnarray}\label{weak_form_diff3a}
&&g(T)+ \frac{\alpha \mu }{2\rho}\tilde{c}_{Ko}\int_0^T\|\pp^2\|_{W^{1,2}(D)}^2+\frac{c}{E}\|\partial^2_{y_1}\xi\|_{L^2(0,L)}^2 dt \\ \nonumber && \qquad \qquad \leq C_1
\int_0^T\|\uu^1\|_{W^{1,2}(D)}^2\|\pp^2\|_{L^2(D)}^2 dt+\vartheta(T).
\end{eqnarray}
Moreover,   (\ref{weak_form_diff3}), (\ref{weak_form_diff3a}) are also valid  for  all fixed $t,\, 0\leq t\leq T$. Therefore after replacing $T$ by $t$ and omitting positive terms in (\ref{weak_form_diff3a}) we obtain
 $$g(t)\leq \vartheta(t)+C_4\int_0^t\|\uu^1\|_{W^{1,2}(D)}^2(s)g(s)ds, \quad C_4=2C_1\alpha^{-1}.$$ 

\noindent
Since $\|\uu^1\|_{W^{1,2}(D)}^2\geq 0$ is an integrable in time and $\vartheta\geq 0$ is a {continuous } non-decreasing function of time,  the Gronwall lemma, see e.g., \cite[Lemma 8.2.29]{FEI}, yields,
\begin{eqnarray}\label{weak_form_diff4}
&&g(t)\leq \vartheta(t)+C_4\int_0^t
\|\uu^1\|_{W^{1,2}(D)}^2(s)\vartheta(s) e^{C_4\int_s^t
\|\uu^1\|_{W^{1,2}(D)}^2(\tau)d\tau}\, ds,\\\nonumber
&&\hphantom{g(t)}\leq\vartheta(t)\left(1+C_4
e^{C_4\int_0^t\|\uu^1\|_{W^{1,2}(D)}^2(s)ds}\int_0^t\|\uu^1\|_{W^{1,2}(D)}^2(s)ds\right)\leq C_5
\vartheta(t).
\end{eqnarray}
Here $C_5$ depends on  $\alpha,\alpha^{-1},\, K,\, (\mu/\rho)^{-1}  \tilde{c}_{Ko}^{-1}$ and on the norms
$\|\uu^i\|_{L^\infty(0,t;L^2(D))}, \\ \|h^i\|_{L^\infty(0,t;H^2(0,L))}, i=1,2$, $\rred \|h^2\|_{W^{1,\infty}(0,t;L^2(0,L))}$, specified above.

Since $\|\pp^2\|^2_{L^2(D)}(t)\leq 2\alpha^{-1}g(t)$, and  $g(t)\leq C_5 \vartheta (t)$, see (\ref{weak_form_diff4}),   
 the right-hand side of (\ref{weak_form_diff3a}) can  now be bounded by
  \begin{eqnarray*}
  && C_1  \int_0^t\|\uu^1\|_{W^{1,2}(D)}^2\|\pp^2\|_{L^2(D)}^2 ds+\vartheta(t)\leq
  \\&&C_12\alpha^{-1}C_5\int_0^t\|\uu^1\|_{W^{1,2}(D)}^2 \vartheta(s)ds +\vartheta(t)\leq \\&&  \vartheta(t)(1+C_12\alpha^{-1}C_5)\int_0^t\|\uu^1\|_{W^{1,2}(D)}^2(s)ds\leq C_6 \vartheta(t), \end{eqnarray*}
and we finally get, cf. (\ref{weak_form_diff3a}),
\begin{eqnarray}\label{estim4}
&&g(t)+\int_0^t \frac{\alpha\mu }{2\rho}\tilde{c}_{Ko}\|\pp^2\|_{W^{1,2}(D)}^2+\frac{c}{E}\|\partial^2_{y_1}\xi\|_{L^2(0,L)}^2 dt  \leq 
C_6\vartheta(t),
\end{eqnarray} where the constant $C_6$ depends on $\alpha,\,\alpha^{-1},\, K,\, (\frac{\alpha\mu }{2\rho}\tilde{c}_{Ko})^{-1}$ and on the norms  of the weak solution $\|\uu^i\|_{L^\infty(0,T;L^2(D))}$ and the data $h^i$
 in ${L^\infty(0,T;H^2(0,L))}\cap W^{1,\infty}(0,T;L^2(0,L)),\ i=1,2$.
Here $g(t), \, \vartheta(t)$ are defined in (\ref{weak_form_diff3}). Thus, with notations from (\ref{pps}) and  $\xi:=E(\eta^1_t-\eta^2_t)$, $\bar{\eta}:=\eta^1-\eta^1$, we proved the following theorem.


\begin{theorem}[Continuous dependence on data]\label{th:contin_dependence}\ \\
Let $(\uu^1,\eta^1),\, (\uu^2, \eta^2)$ be  two weak solutions of the initial boundary value problem (\ref{i.1})--(\ref{i.10}) transformed to the rectangular fixed rectangular domain $D$ satisfying (\ref{weak_form}). Let the corresponding domain deformation be given  by some functions $h^1, \, h^2$ satisfying (\ref{assumptions}).
Let the transformed boundary pressures  $q_{in/out/w}^1$ and $q_{in/out/w}^2$ belong to $L^2(0,T;L^2(S_{in/out/s}))$, respectively.
Then for almost all $t \in (0,T)$ 
it holds:
\begin{eqnarray}\label{contin_dependence}
&&\frac{\alpha}{2}\|\RR\uu^1-\uu^2\|_{L^2(D)}^2(t)+\frac{{\alpha\mu}}{2\rho}\tilde{c}_{Ko}\int_0^t\|\RR\uu^1-\uu^2\|_{W^{1,2}(D)}^2ds
  \\&&\nonumber+\frac{E}{2}\|\eta_t^1-\eta_t^2\|_{L^2(0,L)}^2(t)+\frac{bE}{2}\|\eta^1-\eta^2\|_{L^2(0,L)}^2(t)
  +\frac{aE}{2}\|\eta^1_{y_1}-\eta^2_{y_1}\|_{L^2(0,L)}^2(t)
  \\&&\nonumber\ \ +cE\int_0^t\|\eta_{t}^1-\eta_{t}^2\|_{H^2(0,L)}^2ds 
\\&&
\nonumber \leq C_{7}\left(
\int_0^t\|q_{in}^1-q_{in}^2\|^2_{L^2(0,L)}+\|q_{out}^1-q_{out}^2\|^2_{L^2(0,L)}+\|q_{w}^1-q_{w}^2\|^2_{L^2(0,1)} ds \right.
\\&&\hphantom{\leq C_{8}[}\left.
+{ \omega(t)}\left[\nonumber\|h^1-h^2\|^2_{W^{1,\infty}(0,t;L^2(0,L))}+\|h^1-h^2\|^2_{L^\infty(0,t;H^{2} (0,L))}
\right]\right),
\end{eqnarray}
where \begin{eqnarray*}
 && \omega(t)=\int_0^t\hspace{-2mm}\|\uu^{1}\|_{W^{1,2}(D)}\hspace{-1mm}+\|\uu^{1}\|^2_{W^{1,2}(D)}\hspace{-1mm}+ \|h^1_t\|_{W^{1,\infty}(0,L)}^2\hspace{-1mm}+\|h^2_t\|_{L^\infty(0,L)}^2
  \hspace{-1mm} +{\rred  \|q^1_{w}\|_{L^2{(S_{w})}}^2 }ds\\
  &&\quad \mbox{and} \quad \omega(t)\downarrow 0 \quad  \mbox{for}\quad t\downarrow 0.\qquad 
\end{eqnarray*}
Here $\tilde{c}_{Ko}$ is the coercivity constant of the viscous form coming from the Korn inequality, $\alpha, \, K$ are given by (\ref{assumptions}), $\mu,\rho, E, a, b, c $ are given by the physical model  and  $C_{7}\equiv(C_2+C_3)C_6$ (see the proof above) is a constant depending on $\alpha, \alpha^{-1},K, \, \tilde{c}_{Ko}^{-1}$ and  on the norms  $\|h^i\|_{L^\infty(0,T;H^2(0,L))}$, {\rred $\|h^i\|_{W^{1,\infty}(0,T;L^2(0,L))}$,}  $\|\uu^i\|_{L^\infty(0,T;L^2(D))},$ $i=1,2$.

\end{theorem}


Let us note that the reason for presence of matrix $\RR=J\JJ^{-1}$,  cf. (\ref{matrix_R}), in the above estimate lies in the fact that the domain of definition for solutions $\vv^1$ and $\vv^2$ differs due  to the coupling with the domain deformation. We recall that $\int_{\Omega(h^1)}v^1(X,t)dX=\int_{\Omega(h^2)}v^1(\Phi(x),t)Jdx=\int_{\Omega(h^2)}\tilde{v}^1(x,t)Jdx$, where $\Phi: \, \Omega(h^2)\to  \Omega(h^1)$
and
$ J=det \, \JJ,\ \JJ=\frac{dX}{dx}$.

\medskip

\noindent
\begin{remark}\ \\
  Since $\uu^1-\uu^2=\RR \uu^1-\uu^2+ \EE_\RR\uu^1$, see the lines before (\ref{viscous_diff}), we can write
  \begin{eqnarray*}
    \|\uu^1-\uu^2\|_{W^{1,2}(D)}\leq \|\RR \uu^1-\uu^2\|_{W^{1,2}(D)}+{\rred \|\EE_\RR \uu^1\|_{W^{1,2}(D)}.}
  \end{eqnarray*}
  Since the components of the matrix $\EE_\RR$ contain $\barh$ and $E$, see (\ref{barh_E}), with the assistance of  (\ref{assumptions}) and  Lemma {\ref{lemma:essup}} we finally get\footnote{with analogous consideration as estimates  in (\ref{matrix_E2_est}), (\ref{matrix_E2_est2}).}
  \begin{eqnarray*}
   {\rred \|\uu^1-\uu^2\|_{L^{2}(D)} \leq \|\RR \uu^1-\uu^2\|_{L^{2}(D)}+C_{K, \alpha}\|\barh\|_{W^{1,\infty}(0,L)}\|\uu^1\|_{L^{2}(D)}\|h^1\|_{W^{1,\infty}(0,L)},}\\
    \|\uu^1-\uu^2\|_{W^{1,2}(D)} \leq \|\RR \uu^1-\uu^2\|_{W^{1,2}(D)}
+C_{K, \alpha}\|\barh\|_{W^{1,\infty}(0,L)}\|\uu^1\|_{W^{1,2}(D)}\hspace{1.5cm}
    \\\hspace{4cm}+C_{K, \alpha}(\|\barh\|_{H^2(0,L)}+\|\barh\|_{L^\infty(0,L)}\|h^1\|_{H^2(0,L)})\|\uu^1\|_{W^{1,2}(D)}.
  \end{eqnarray*}
  Hence for $h^1\to h^2$ {\rred in $W^{1,\infty}(0,T;L^2(0,L))\cap H^1(0,T;H^{2}_0(0,L)) $} and $g^1_{in/out/w}\to g^2_{in/out/w}$ {\rred in $L^2(0,T;L^2(\partial D))$}  
    the above estimates and  (\ref{contin_dependence}) imply that \\$ \|\uu^1-\uu^2\|_{L^2(0,T;W^{1,2}(D))\cap L^\infty(0,T;L^{2}(D)) }\to 0$. Thus the weak solution is  continuously dependent  on  the domain deformation and  the boundary pressure. 
 \end{remark}

\begin{corollary}[Uniqueness of the  weak solution]
\label{col:uniqueness1}\ \\
If the boundary conditions for the pressure as well as the boundary deformation coincide for both solutions, i.e., $q_{in}^1=q_{in}^2,\, q_{out}^1=q_{out}^2,\, q_{w}^1=q_w^2, \, h^1=h^2:=h$, then there exists a unique weak solution to the problem  (\ref{i.1})--(\ref{i.10})  on $\Omega(h)$, i.e., to the problem  linearized with respect to the geometry in the sense of (\ref{assumptions})--(\ref{timederivativeproperty2}). 
\end{corollary}
\begin{proof}
  The proof is a consequence  of Theorem~\ref{th:contin_dependence}. Note that for  $h^1=h^2$ we have $\RR=\II$ in (\ref{contin_dependence}).
  \end{proof}




\section{Fixed point procedure \\with respect to the geometry}
\label{sec:fixed_point}

In the previous section we have proven uniqueness of the weak solution $(\uu,\eta)$ defined on $\Omega(h)$ given by some sufficiently smooth functions $\delta, R_0$,  $ h:=R_0+\delta$ satisfying (\ref{assumptions}). This allows us to consider the following fixed point procedure with respect to the geometry of the computational domain.
We consider a  mapping $\cal F$ that is well defined through the weak formulation on $\Omega(h)$ (or its transformed form (\ref{weak_form})). It maps the given function $\delta$ to the unique weak solution $\eta$,   ${\cal F}(\delta)=\eta$. Once there exists a fixed point of this mapping, the original fluid-structure interaction problem (\ref{i.1})--(\ref{i.10}) is solved. Let us point out that  the Schauder fixed point argument yields the existence (but not uniqueness) of the weak solution to similar problems, see \cite{GRA05, GRA08, LENRU, HLN14, HLN14book}. Uniqueness for similar problems for  Newtonian fluids has been shown e.g.  in \cite{PAD10} by deriving an additional estimate on the continuous dependence on the initial data.

In the following procedure we apply the Banach fixed point theorem and obtain a unique fixed point  $\eta^\star={\cal F}(\eta^\star)$.
Using  this fixed point argument we prove at the same time the convergence of the iterative process  $\eta^k={\cal F}(\eta^{k-1})$, i.e. the convergence of the global iterative method with respect to the domain geometry. Additionally, Theorem \ref{th:contin_dependence} implies  uniqueness of $\uu$, thus  we finally obtain uniqueness of the (complete) weak solution $(\uu,\eta)$ obtained by this iterative procedure.
\medskip

\noindent
We consider following iterative process with respect to $\eta$;
\\Let $\delta=\eta^{k-1}, k=1,2, \ldots,\,$ with given $ \eta^0$,  e.g., $ \eta^0=0$.  Let $(\uu^k, \eta^k)$ be the  weak solution obtained on $\Omega(h)=\Omega(R_0+\eta^{k-1})$, i.e.,  $(\uu^k, \eta^k)$ satisfies (\ref{weak_form}) for $h=R_0+\eta^{k-1}$.

We define the space $\ZZ:=H^1(0,T;H_0^2(0,L))\cap W^{1,\infty}(0,T;L^2(D))$ and the mapping 
 \begin{eqnarray}\label{mapping_F}
{\cal F}:B_{\alpha}\subset \ZZ \to \ZZ,\ \ \ {\cal F}(\eta^{k-1})=\eta^k.
\end{eqnarray}
 Here $B_{\alpha}$ is a ball in  space $\ZZ$ chosen such that the necessary properties of the domain deformation from (\ref{assumptions}), avoiding the contact of the deforming wall with the fixed bottom boundary 
\footnote{The condition $\alpha\leq R_0+\eta\leq \alpha^{-1}$ in (\ref{assumptions}) avoids the contact of the moving wall with the fixed bottom boundary for $\alpha >0$.}, are  guaranteed. Let us specify $B_{\alpha}$ in more details.
 Let $0<\alpha<1, \, R_0\in C^2[0,L], \, 0<R_{\min}\leq R_0\leq R_{\max},$ $\{R_{\min},R_{\max}\}\in \RR^+$ be given. Moreover let $\alpha$ be such that
 \begin{equation}\label{prop_alpha}
   \alpha<\min  \{ R_{\min},\, ( {R_{\min}+R_{\max}})^{-1} \}.
 \end{equation}
 We define
 \begin{eqnarray}
B_{\alpha}=\{\eta\in \ZZ\, s.t.\ \eta(y_1, 0)=\eta_t(y_1, 0)=\eta_{y_1}(y_1, 0)=0,\  \|\eta\|_\ZZ\leq  R_{\min}-\alpha \}.\ \ 
\end{eqnarray}
 In what follows we show that each $\eta \in B_\alpha$  satisfies  properties  (\ref{assumptions}) for sufficiently small time $T$ and small boundary pressures. Indeed, due to $\ZZ\subset C(0,T;C^1[0,L])$ and the zero initial conditions there exists a time $T_{\alpha}$ such that 
 it holds for each $\eta \in B_\alpha$: $$ \max_{0\leq t\leq T_\alpha,\, 0 \leq y_1\leq L}|\eta(y_1,t)|=\|\eta\|_{C([0, T_\alpha]\times[0,L])}\leq R_{\min}-\alpha.$$
 Then the condition $\alpha\leq R_0+\eta \leq \alpha^{-1}$ is satisfied for each $\eta\in B_\alpha$ and for all $t\leq T_\alpha$, where  $\alpha$ is given by property (\ref{prop_alpha}). 
\\ Moreover since $\eta_t(y_1, 0)=\eta_{y_1}(y_1, 0)=0 \ \forall y_1\in [0,L]$, there exists a maximal time $T_K$ such that $$\int_0^{T_K}|\eta_t(y_1,s)|^2ds+|\eta_{y_1}(y_1,t)|^2\leq K\ \ \forall y_1\in[0,L], \ \forall t \in [0,T_K].$$ Thus, each function from  $ B_\alpha$   fulfills both properties (\ref{assumptions}) for all $y_1 \in [0,L]$ and sufficiently small time $\tilde{T}=\min\{T_\alpha, {T_K}\}$ and  it is an admissible function for the domain deformation in view of our approximation.

   In order to apply the Banach fixed point theorem to our mapping $\cal F$ we have to verify  that this mapping is ``onto", i.e., ${\cal F}(B_\alpha) \subseteq B_\alpha$. Indeed for $\eta^{k-1}\in B_\alpha$ the  weak solution obtained using $h=R_0+\eta^{k-1}$
   satisfies the following a priori estimate shown in \cite[Section 5, cf. (5.1) and  (4.7)]{HLN14},
   \begin{eqnarray}
     &&\|\uu^k\|^2_{L^\infty(0,T;L^2(D))}+\|\uu^k\|^2_{L^2(0,T;W^{1,2}(D))}\nonumber
     \\&&+\|\eta^k\|_{L^\infty(0,T;H^1(0,L))}^2+\|\eta^k_t\|_{L^\infty(0,T;L^2(0,L))}^2+ \|\eta^k_t\|_{L^2(0,T;H^2(0,L))}^2\nonumber
     \\&&\leq C(\alpha, \alpha^{-1})(1+{\cal H}e^{\cal H})\left(T\|R_0\|^2_{C^2[0,L]}+\int_0^T\|q_{\partial_D}\|^2_{L^2(\partial D)}dt\right),\label{apr_est}
     \end{eqnarray}
   where ${\cal H}\leq \int_0^T|\eta_t^{k-1}|+| \eta_t^{k-1}|^2dt$ for all $ y_1 \in[0,L]$
   and $\|q_{\partial_D}\|^2_{L^2(\partial D)}$ represents the sum of the norms of the boundary pressures, i.e.,  $\|q_{\partial_D}\|^2_{L^2(\partial D)}:=\|q_{in}\|^2_{L^2(S_{in})}+\|q_{out}\|^2_{L^2(S_{out})}+\|q_{w}\|^2_{L^2(S_{w})}$.
   Note, that due to  (\ref{assumptions}) which is fulfilled by all $\eta^k, \, k=1,2,\ldots,$ \ for  $\tilde{T}$ specified above, ${\cal H}$ is also uniformly bounded with  ${\cal H}\leq\sqrt{{T} K}+K$ for a final time $T\leq \tilde{T}$.  Thus the right-hand side of (\ref{apr_est}) can be bounded uniformly in $k$. Moreover,  (\ref{apr_est})  implies that
   $$\|\eta^k\|^2_\ZZ\leq  C(\alpha, \alpha^{-1})(1+(\sqrt{{T}K}+K)e^{(\sqrt{{T}K}+K)})\left({T}\|R_0\|^2_{C^2[0,L]}+\int_0^{{T}} \hspace{-2mm}\|q_{\partial_D}\|^2_{L^2(\partial D)}\right)$$
   for $T\leq \tilde{T}$.
   Therefore, for  given $\alpha, \, K$ 
   and  sufficiently small  time $T=:T_{\max}\leq \tilde{T}$  we get  $\|\eta^k\|_\ZZ\leq R_{\min}-\alpha$. Consequently,  ${\cal F}(B_\alpha) \subseteq B_\alpha$ for small enough final time $T_{\max}$. 

\medskip
   
   The next step  is to verify the contractivity of the mapping $\cal F$, i.e.,  to show for $\delta^1\neq \delta^2$,  $\delta^1:=\eta^{1,\, k-1}, \, \delta^2:=\eta^{2,\, k-1}$ and ${\cal F }(\delta^1)=\eta^{1,\, k}, \, {\cal F }(\delta^2)=\eta^{2,\,  k}$ that
   \begin{equation}\label{contractevness}
     \|{\cal F }(\delta^1)-{\cal F }(\delta^2)\|_\ZZ  \leq q \|\delta^1-\delta^2\|_\ZZ,\ \quad \mbox{where}\  \quad q <1.
     \end{equation}
   Theorem \ref{th:contin_dependence} from the previous section provides the essential estimate to prove the contractivity.  Indeed, (\ref{contin_dependence}) implies for $q_{in}^1=q_{in}^2,\, q_{out}^1=q_{out}^2,\, q_{w}^1=q_w^2$   that
   \begin{eqnarray*}
   &&  \|\eta^{1, \, k}-\eta^{2, \, k}\|^2_{W^{1,\infty}(0,t;L^2(D))\cap H^{1}(0,t;H^2(D)) }
    \\&& \qquad \leq C(E,a,b,c)\, C_7\, \omega(t)\|\eta^{1,\,  k-1}-\eta^{2, \, k-1}\|^2_{W^{1,\infty}(0,t;L^2(D))\cap H^{1}(0,t;H^2(D)) },
     \end{eqnarray*}
   where $\omega(t)\downarrow 0$ as $t \downarrow 0$. The constant $ C(E,a,b,c)$ depends only on physical data, whereas the constant $C_7$ depends on $K, \alpha, \alpha^{-1}$, on $C_h$, cf. (\ref{C_h}),   on $\tilde{c_h}$, cf. (\ref{boundary_estim}) and on $C_\varepsilon^{\RM{1}-\RM{4}}$, thus  $C_7$ contains
    positive powers of  $\|\eta^{i,\ k-1}\|_{L^\infty(0,T;H^2(0,L))}$, {\rred $\|\eta^{i,\ k-1}\|_{W^{1,\infty}(0,T;L^2(0,L))}$,}  
$\|\uu^{i,\ k}\|_{L^\infty(0,T;L^2(D))},\  i=1,2$.
    Estimate (\ref{apr_est}) and analogous consideration as above  imply $\|\uu^{i,\ k}\|_{L^\infty(0,T;L^2(D))}\leq {R_{\min}-\alpha}$ for  sufficiently small $T\leq T_{\max}$ as well as $C_h\leq1+2({R_{\min}-\alpha})$ and 
    $\tilde{c}_{h}\leq(1+({R_{\min}-\alpha})$.  Thus  $C_7$ can be bounded from above with some constant independent of $k$. 

 In view of this considerations  (\ref{contractevness}) holds with $q=C(E,a,b,c)C_7\omega(t)$  for sufficiently small $t\leq T_{\max}$  such that $\omega(t)<\frac{1}{C(E,a,b,c)C_7}$. We have shown, that the mapping $\cal F$ is a contraction.

 \medskip
 
 Finally, the Banach fixed point theorem, see, e.g., \cite[Theorem 1.5]{RUZF}, implies that  for
 sufficiently small time  and $ 0<\alpha<1$ satisfying (\ref{prop_alpha})  there exists one and only one fixed point $\eta^\star$ of the mapping $\cal F$ defined in (\ref{mapping_F}), $\eta^\star={\cal F}(\eta^\star)$. Moreover, the  iterative process $\eta^k={\cal F}(\eta^{k-1})$ converges, i.e., $ \eta^k\to \eta^\star$ in $\cal Z$. 
   The continuous dependence of the fluid velocity $\uu$ on the domain deformation (Theorem \ref{th:contin_dependence})  and
  {\rred Corollary \ref{col:uniqueness1} applied for $h^1=R_0+\eta^\star, h^2=R_0+\eta^k$ imply moreover
    the convergence of $\uu^k\to \uu^\star$ in  $L^2(0,T;W^{1,2}(D))\cap L^\infty(0,T;L^{2}(D))$ and the uniqueness of the  weak solution $(\uu^\star,\eta^\star)$  defined on $\Omega(\eta^\star)$.}

  Thus, we have proven that the global iterative method with respect to the domain deformation, used in our numerical approximation of the problem  (\ref{i.1})--(\ref{i.10}), see  \cite{LZ2008,LZ2010, HLR2012},  converges at least
  for sufficiently small time.
 {\red  In this construction of the solution the contact of the moving wall with the fixed bottom boundary is avoided.}
 
\subsection*{Concluding remarks}

In  \cite{HLN14}, \cite{HLN14book} we have studied the existence of weak solution to the similar fluid-structure interaction problem for shear-dependent fluids with viscosity obeying  the power-law $\mu(|e(\vv)|)= \mu{(1+|e( \vv)|^2)^{\frac{p-2}{2}}}$ for $p\geq 2$. In that case the corresponding space of weak velocities is  $L^p(0,T;W^{1,p}(D))$.
Therefore, the norms  $\|h^1-h^2\|$ on the right-hand side of  estimate (\ref{contin_dependence}) would appear  with powers $2$ and $p' <2$, where $\frac1{p}+\frac1{p'}=1$. This leads to a difficult problem, since such an estimate would  only imply
$\|{\cal F }(\delta^1)-{\cal F }(\delta^2)\|_\ZZ  \leq q \|\delta^1-\delta^2\|^{p'/2}_\ZZ,\, q<1$, compare (\ref{contractevness}). This estimate does not imply the Cauchy-property of the sequence $\eta^k$, that is essential 
for the limiting process $k\to \infty$ 
in the proof the Banach fixed point theorem and consequently for the existence of a fixed point.

Another challenge is to  show the continuous dependence of the weak solution on initial data
for a shear-dependent power-law fluid and to generalize  the result for Newtonian fluids from \cite{PAD10}. 
In this case, again, due to the $L^p$-structure  of the fluid velocity, the estimate of the viscous term may cause a difficulty, since  except of the $p$-th power of the gradients of the weak solution it also contains  the term $\partial^2_{y_1}(\hfrac)$, cf. (\ref{matrix_E2}), that is only bounded in $L^2(0,L)$. Thus, more regularity  of the weak solution would be necessary for the corresponding estimates. In this case  the so called weak - strong uniqueness would be  an interesting tool to study such fluid-structure interaction problems for power-law fluids. This may be a goal of our future work.

 \bigskip

 \noindent
 {\bf{Acknowledgment}}\\
 The author has been supported 
 by the German Research Foundation DFG, Project No.  HU 1885/1-2.


\bigskip

 \appendix

 
 \section{Compactness argument for the  $(\kappa, \veps)$-approximate weak solution}

 \label{appendix}

 The existence of a weak solution to the problem  (\ref{i.1})--(\ref{i.10}) on the deformable domain $\Omega(h)$ moving according to the a priori given function $h(y_1,t)$, cf. (\ref{assumptions}),  for  power-law fluids has been studied in our previous work \cite{HLN14} by transformation to the fixed rectangle domain $D$. See also \cite{FZ10} for similar result for Newtonian fluid.
 We have applied the artificial compressibility method in order to handle the  divergence-free
 condition and regularized it  using the parabolic equation for pressure $\veps \partial_t p-\veps \Delta p+ \dive \vv=0$. Moreover,  the fluid-structure coupling condition  on the moving boundary (\ref{i.55}), (\ref{dirichlet}) has been approximated  by introduction of the semi-pervious boundary, see  \cite[(2.6)--(2.7)]{HLN14} for more details. Condition (\ref{dirichlet}) has been replaced by $\kappa(\partial_t \eta-v_2|_{\Gamma_w})$, where $\kappa$ is a penalization parameter; the original  coupling condition is satisfied for $\kappa = \infty$.

 Taking $\kappa=\veps^{-1}$, the passage to the limit in the $(\kappa, \veps)$-approximate weak solution $(\uu_\kappa, q_\kappa,\sigma_\kappa)$, $\uu_\kappa:=\uu_{\kappa,\veps}, \, q_\kappa:= q_{\kappa,\veps},\, \sigma_\kappa:=\partial_t\eta_{\kappa,\veps}$
 for  $\veps \to 0, \ \kappa \to \infty $ has been performed  at once. 
 Here  $q_{\kappa,\veps}(y,t)= p_{\kappa,\veps}(x,t)/\rho$ denotes the pressure transformed to fixed domain $ D$.  For the limiting process in our $(\kappa, \veps)$-approximate weak formulation the strong convergence
 $\uu_\kappa \to \uu$ 
 is necessary.  
Since the a priori estimates for $\partial_t \uu_{\kappa}$ depend on $\veps, \, \kappa$, the classical  Lions-Aubin compactness argument cannot be applied.
  The strong convergence for similar artificial compressibility approximation on fixed domains has been shown in  \cite[Ch. III., Th. 8.1]{TEMAM} by the  compactness argument involving fractional time
derivatives \cite[Ch. III., Th. 2.2]{TEMAM} and the Fourier transform. However, due to difficulties related to the moving domain and transformed  differential operators depending on time,  this approach seems to be inappropriate  for deformable domains. 

In our approach we  used  a compactness criterion based on integral equicontinuity in time \cite[Lemma 1.9]{AL}. {\rred The integral equicontinuity estimates have been shown in our previous works, see \cite[(5.4)]{HLN14}, \cite[(8.4)]{FZ10}. However,  some
  terms containing the 
  velocity divergence 
  has not been treated carefully in the mention works, 
 the dependence of equicontinuity  estimates 
on  $\veps$ has been ignored. 
}

\medskip

{\rred In the following} lines we present a corrected proof of those  integral equicontinuity estimate   for
$\uu_\kappa$, 
which is {independent on $\veps$}.
To this end  we  however shall simplify our artificial  compressibility condition and consider only elliptic equation for pressure, $$-\veps \Delta p_\kappa+\dive \vv_\kappa=0\ \mbox{in}\ \Omega(h),\ \nabla p_\kappa \cdot\nn=0 \  \mbox{on}\ \partial \Omega(h), \ {\rred\int_{\Omega(h)}p_\kappa=0}$$ in our  $(\kappa, \veps)$-approximation of (\ref{i.1})--(\ref{i.10}).
{\rred Note, that the proof of existence of weak solution for fixed $\varepsilon,\, \kappa$ with only the elliptic (instead of parabolic) compressibility approximation can be obtained analogously as the proof in  \cite[Section 4]{HLN14}, compare also \cite{FZ10}, using the same  techniques.
  Here 
  the coercivity of the bilinear form arising in the pressure equation 
is now guaranteed due to the  average condition $\int_{\Omega(h)}p_\kappa=0$ (Poincar\'e-Wirtinger inequality). 
    Repeating the estimation process from  \cite{HLN14} we obtain the first a priori estimates for $\uu_\kappa(y,t)=\vv_\kappa(x,t)$ and $\eta_\kappa(x_1,t)$. For the pressure one  gets now $p_\kappa(x,t)=q_\kappa(y,t)\in L^2(0,T;H^1(D))$. These estimates are, similarly as in  \cite{HLN14, FZ10} independent on $\kappa, \veps$. Using the pressure equation  (i.e. considering zero test function for $\uu_\kappa, \eta_\kappa$) we obtain by letting $\kappa\to \infty$,  that the weak limit of $\uu_\kappa$ is divergence free almost everywhere, compare  \cite[Section 5]{HLN14} 
}.

Now we concentrate  on the proof of  the compactness of $\uu_\kappa$. In what follows we show  the integral equicontinuity for $\uu_\kappa$, which is independent on $\veps$. Note, that terms coming from the divergence of the velocity  and  the pressure gradient  are bounded  by a priori estimates with $c/\sqrt{\veps}$ for some constant $c$, i.e., not uniformly in $\veps$.  The crucial idea in our new proof is to eliminate these terms by choosing appropriate test functions involving the so called Piola transformation of the solution. 

\medskip

The corresponding  $(\kappa, \veps)$-approximate weak formulation  of the coupled problem with elliptic compressibility approximation reads as follows:
\begin{eqnarray}\label{WF1}
  &&\int_0^T
- \Big \langle \partial_t (h\uu_\kappa),\pp \Big \rangle 
   \, dt=\\
&&
  \int_0^T\bigg\{\int_D -\frac{\partial h}{\partial t}\frac{\partial (y_2\uu_\kappa )}{\partial y_2}\cdot \pp
  -hq_\kappa \dive_h \pp+ h \dive_h\uu_\kappa \phi\, dy
\nonumber  \\&&\hphantom{=\int_0^T\ }
+\fl\uu_\kappa , \pp \fr_h\, + \, b_{h}(\uu_\kappa , \uu_\kappa , \pp )+\veps a_1(q_\kappa,\phi) \nonumber \\
&&
\hphantom{=\int_0^T\ }
{ +}\int_0^1 h(L) q_{\it out}(y_2,t)\psi_1\, (L,y_2,t)
-h(0) q_{\it in}(y_2,t)\psi_1\,(0,y_2,t)\,dy_2 \nonumber \\
&&
\hphantom{=\int_0^T\ }
+\int_0^L\left(
q_{\it w}+\frac{1}{2}\frac{\partial h}{\partial t}u_{2,\kappa}+\kappa(u_{2,\kappa}-\sigma_\kappa)
\right)\psi_2\,(y_1,1,t)dy_1
\nonumber\\
&&
\hphantom{=\int_0^T\ }
+\int_0^L\frac{\partial\sigma_\kappa }{\partial t} \xi 
+c\,\frac{\partial^2 \sigma_\kappa}{\partial y_1^2}\frac{\partial^2\xi}{\partial y_1^2}
+a\frac{\partial}{\partial y_1} \left(\int_0^t\sigma_\kappa(y_1,s)ds\right)\frac{\partial\xi}{\partial y_1}(y_1,t)  \nonumber \\
&&\hphantom{+\int_0^T\int_0^L \left(\right.}\hspace{2mm}
-{ a}\frac{\partial^2 R_0}{\partial y_1^2}\xi
+b\int_0^t \sigma_\kappa(y_1,s)ds\, \xi (y_1,t)+\frac{\kappa}{E}(\sigma_\kappa-u_{2,\kappa})\xi\,dy_1\,\bigg\}\,dt \;  \nonumber
\end{eqnarray}
 for every
$(\pp ,\phi , \xi)\in H_0^1(0,T;\VV)\times L^2(0,T;H^1(D))\times
L^2(0,T;{ H^2_0}(0,L))$, where $\VV\equiv\{W^{1,2}(D);\ w_1=0 \ \mbox{on}\ S_w,\  w_2=0 \ \mbox{on}\ S_{in}\cup S_{out}\cup S_c \}$, cf. (\ref{eij}), (\ref{4.3}) 
  and the definition of the form $a_1(\cdot,\cdot)$ in \cite[(2.10)]{HLN14}.

  \medskip
 
  Let us first specify  appropriate test functions in order to obtain the integral equicontinuity estimate. 
  By similar procedure as in \cite[p. 221-222]{HLN14}, i.e. considering test functions $\chi_\delta(s)(\pp,\phi,\xi)$, where $\chi_\delta(s)$ is a smooth approximation of the  characteristic function of the interval $(t,t+\tau)$,  $\pp=\pp(y,s)$ is now a time-dependent function,  $\phi(y), \xi(y_1) $ are time-independent, and letting the smoothness parameter $\delta \to 0$
  we obtain from (\ref{WF1}) 
\begin{eqnarray}\label{WF2}
 && \hspace{0mm}-\int_0^{T-\tau}
 \hspace{-2mm}\int_D  
[ (h\uu_\kappa\psi)(y,t+\tau)- (h\uu_\kappa\psi)(y,t)]dy\, 
\\&&\hspace{13mm}+\int_0^L \hspace{-2mm} [\sigma_\kappa (y_1,t+\tau)- \sigma_\kappa(y_1,t)]\xi(y_1)dy_1\,  dt\nonumber
 \\\ &&\hspace{0mm}=\int_0^{T-\tau}\hspace{-2mm}\int_t^{t+\tau} \hspace{-2mm} \int_D \left \{h(y_1,s)\dive_{h(s)}\uu_\kappa(y,s)\phi(y)
     -(h q_\kappa)(y, s)\dive_{h(s)}\pp(y,s)\nonumber\right.
\\ &&\hphantom{\hspace{6cm}}\left.- h\uu_\kappa(y,s)\cdot {\partial_t \pp(y,s)} +\ldots \right\}\, dy \, ds\, dt.\nonumber
   \end{eqnarray}
Note that in our case  the operator of divergence depends on $h(y_1,t)$, thus on time. 
In order to commute the integral $\int_t^{t+\tau}ds$ with the $\dive_h$ - operator and to eliminate the divergence terms on the right-hand side of (\ref{WF2})  we rewrite them   in terms of the time independent $\dive$ - operator using the Piola transformation.

We consider the Piola transformation $\vv_P(y,t): D\times[0,T]\to \RR^2$ of our velocity field $\vv: \Omega(h)\times[0,T]\to \RR^2$, compare e.g., \cite[Chapter II]{SINI}
$$\vv_P(y,t):=det (\nabla {\cal L}(y,t)) {\nabla \cal L}^{-1}(y,t)\vv(x,t)
,$$
where  the mapping ${\cal L}(y,t)\stackrel{\mbox{\footnotesize def}}{=}(y_1, y_2 h(y_1,t))=x(t)\in \Omega(h)$ describes the transformation of variables between $\Omega(h)$ and $D$,  
\begin{eqnarray}\label{matrices}
  {\nabla \cal L}=\frac{dx}{dy}=
  \begin{bmatrix}1 & 0\\y_2 \partial_{ y_1}{h}& {h}\end{bmatrix}=:\JJ,
\quad {\nabla \cal L}^{-1}=\begin{bmatrix}1 & 0\\-\frac{y_2}{h}\partial_{ y_1}{h}& \frac{1}{h}\end{bmatrix}=:\JJ^{-1}, 
\end{eqnarray} 
compare also (\ref{matrix_R})  for $h^1\equiv h,h^2 \equiv 1.$  Replacing $\vv(x,t)$ by $\uu(y,t)$ we rewrite the Piola transformation as
$$\vv_P(y,t)=J \JJ^{-1}(y,t)\uu(y,t)=\RR(y,t)\uu(y,t),$$
where $\RR:= J \JJ^{-1}$ and $J:=det(\nabla {\cal L})=h$.  Note that since $\JJ^{-1}=J^{-1} cof \JJ^T$ we have $\RR=cof \JJ^T$.
Now let us denote the divergence operators with respect to fixed and moving coordinates as $\dive \, \uu (y,t):=\partial_{y_1}u_1+\partial_{y_2}u_2, \ \dive_x \vv (x,t):=\partial_{x_1}v_1+\partial_{x_2}v_2,  $ respectively. 
From the Piola identity: $\dive \,( cof \nabla {\cal L})=0$  and the transformation of  the differential operator
\footnote{$\dive \,\vv_P(y,t)=\dive (\RR\uu(y,t))=0+(\RR^T\nabla_y)\cdot \uu(y,t)=(\RR^T\nabla {\cal L}^T\nabla_x)\cdot \vv(x,t) =J(y,t) \dive_x \vv (x,t).$}
it follows that $$\dive\, \vv_P(y,t)=J(y,t) \dive_x\, \vv (x,t), $$ 
i.e., divergence  of the velocity with respect to time dependent coordinates $x(t)$
can be expressed by means of the divergence  in time-independent coordinates $y$ of its Piola transformation. Thus,
\begin{equation}\label{Piola}
  J(y,t) \dive_x\, \vv (x,t) \Big[\equiv h(y_1,t) \dive_{h(y_1,t)}\uu(y,t)\Big] =\dive\,(\RR(y,t)\uu(y,t)).
\end{equation}
With the assistance of  (\ref{Piola}) we obtain for the divergence terms on the  right-hand side of (\ref{WF2}),
\begin{equation}\label{div_terms2}
  \int_0^{T-\tau} \int_D \Big [\dive\left(\int_t^{t+\tau} (\RR\uu_\kappa)(y,s)ds\right)\phi(y)
    -\int_t^{t+\tau} \hspace{-2mm} q_\kappa(s)\dive\, (\RR\pp)(y,s)ds\Big ]\, dy \,  dt.
\end{equation}
Now, let us consider the following test functions
\begin{eqnarray}\label{testfunctions1}
  \pp(y,s)&=&\RR^{-1}(y,s)\left[(\RR \uu_\kappa)(y,t+\tau)-(\RR \uu_\kappa)(y,t)\right],
  \\\nonumber \phi(y;t)&=&q_\kappa(y,t)- q_\kappa(y,t+\tau),
  \\\xi(y_1;t)&=& {\sigma_\kappa}(y_1,t+\tau)-{\sigma_\kappa}(y_1,t),\nonumber
\end{eqnarray}
where $s\in [t,t+\tau]$ and $t\in [0,T-\tau] $ is fixed.
Let us insert the test functions (\ref{testfunctions1}) into  (\ref{div_terms2}). Using the following notations
$$U(t):=\int_t^{t+\tau}\RR(s)\uu_\kappa(s)ds,\quad V(t):=\int_t^{t+\tau}q_\kappa(s)ds,$$
  (\ref{div_terms2}) can be rewritten as
\begin{eqnarray*}
& -& \int_0^{T-\tau}\hspace{-2.5mm} \int_D \hspace{-1mm}\Big [V'(t)\dive\, U(t)+V(t)\dive\,U'(t)\Big]dydt\\
 && = \int_D\hspace{-2mm}  V(0)\dive\, U(0)-V(T-\tau)\dive\, U(T-\tau)\,dy\\
&&=\int_D\int_0^\tau \hspace{-2mm}q_\kappa ds\int_0^\tau\hspace{-2mm}\dive (\RR\uu_\kappa)ds-\int_{T-\tau}^{T}\hspace{-3mm} q_\kappa ds\int_{T-\tau}^T\hspace{-3mm}\dive (\RR\uu_\kappa)ds\, dy.
\end{eqnarray*}

\noindent
According to the above considerations using the test functions  (\ref{testfunctions1})
we finally obtain from  (\ref{WF2}), 
\begin{eqnarray}\label{WF3}\nonumber
 \hspace{-7mm} &&\int_0^{T-\tau}
 \hspace{-3mm} \int_D \hspace{-2mm}  \big[(h\uu_\kappa)(t+\tau)\cdot \RR^{-1}(t+\tau)  -(h\uu_\kappa)(t)\cdot \RR^{-1} (t) \big]\big[(\RR\uu)(t+\tau)-(\RR\uu)(t)\big]  dy 
 \\ \hspace{-7mm} &&\qquad +\int_0^L [\sigma_\kappa(t+\tau)-\sigma_\kappa(t)]^2 \, dy_1\, dt\ =\nonumber
 \\ \hspace{-7mm} &&- \underbrace{ \int_D\int_0^{\tau}\hspace{-2mm} q_\kappa(s)\,ds \int_0^{\tau} h(s)\dive_{h(s)}\uu_\kappa(s)\, ds}_{I}
+\underbrace{\int_{T-\tau}^{T}\hspace{-3mm}q_\kappa(s)\,ds\int_{T-\tau}^{T} \hspace{-2mm}h(s)\dive_{h(s)}\uu_\kappa(s)\, ds\, dy}_{II} \nonumber
\\ \hspace{-7mm} &&+ \underbrace{\int_0^{T-\tau}\hspace{-2mm}\int_t^{t+\tau}\hspace{-2mm}\int_D  h\uu(s)\cdot (\partial_t\RR^{-1}(s))\big[(\RR\uu)(t+\tau)-(\RR\uu)(t)\big]\, dy\, ds}_{III}\, dt +\ldots  \hspace{-0mm} 
\end{eqnarray}


In what follows we  show, that  the terms $I,\, II$ on the right-hand side of (\ref{WF3}) can be estimated with $\tau C_{K,\alpha}$ independently on $\veps$,  where $C_{K,\alpha}$ is some constant dependent on $\alpha, K$, cf. (\ref{assumptions}).
We will demonstrate this estimate for term $I$, term $II$ can be estimated analogously. 
Using the equation for pressure, i.e., inserting  test functions $(\oo,\phi,0)$ into (\ref{WF1}), integrating it over $\int_0^\tau ds$ instead of  $\int_0^T dt$,  where $\tau <T$ is fixed, and using time-independent test function $\phi(y\,;\tau)=\int_0^\tau q_\kappa(y,s)ds$  we obtain
$$I=\veps\int_0^\tau a_1(q_\kappa(y,s),\phi(y;\,\tau)) ds.$$
From  the definition of $a_1(\cdot,\cdot)$, see \cite[(2.10)]{HLN14} and (\ref{assumptions})
it follows $$I \leq C_{K,\alpha} \int_0^\tau\|\sqrt{\veps}\nabla q_\kappa(s)\|_{L^2(D)}ds\Big\|\int_0^\tau\sqrt{\veps}\nabla q_\kappa (s)ds \Big\|_{L^2(D)}.$$
Since  by the H\"older inequality we have $$\int_0^\tau\|\sqrt{\veps} \nabla q_\kappa\|_{L^2(D)}ds\leq \tau^\frac1{2}\|\sqrt{\veps}q_\kappa\|_{L^2(0,T;H^1(D))},$$ 
$$\Big\|\int_0^\tau \sqrt{\veps}\nabla q_\kappa (s)ds \Big\|_{L^2(D)}^2 \leq \tau \int_D\int_0^\tau|\sqrt{\veps}\nabla q_\kappa|^2 (y,s)ds \,dy
\leq \tau \|\sqrt{\veps}q_\kappa\|_{L^2(0,T;H^1(D))}^2,$$
we finally obtain
\begin{equation}\label{est1}
I\leq \tau C_{K,\alpha} \|\sqrt{\veps} q_\kappa\|_{L^2(0,T;H^1(D))}^2\leq c\tau.
\end{equation}

\medskip

\noindent
Now we rewrite the first term on the left-hand side of (\ref{WF3}) by means of square of time differences. To this end we first  rewrite it as
\begin{eqnarray}\label{timeder_1}
  \int_0^{T-\tau}\hspace{-3mm}
\int_D \hspace{-2mm} \big[(\RR^{-T}h\uu_\kappa)(t+\tau)  -(\RR^{-T}h\uu_\kappa)(t) \big]\cdot\big[(\RR\uu_\kappa)(t+\tau)-(\RR\uu_\kappa)(t)\big] dy\, dt \  \ \  
\\  =\int_0^{T-\tau}\hspace{-2mm}
\int_D  \big[(\JJ^{T}\uu_\kappa)(t+\tau)  -(\JJ^{T}\uu_\kappa)(t) \big]\cdot\big[(\RR\uu_\kappa)(t+\tau)-(\RR\uu_\kappa)(t)\big]  dy\, dt,\ \nonumber
\end{eqnarray}
here we have used $\RR^{-T}=h^{-1}\JJ^T$.
\\Further, since $\JJ^T\uu_\kappa\cdot\RR\uu_\kappa=\RR^T\JJ^T \uu_\kappa\cdot \uu_\kappa=h|\uu_\kappa|^2$,
 we  can rewrite  equality (\ref{timeder_1})  in terms  of  square of  the time difference of $\sqrt{h}\uu_\kappa$ and an additional term $ IV$ as follows.
 \begin{eqnarray}\label{timeder_2}
 \hspace{-5mm}&&\int_0^{T-\tau} \hspace{-2mm}\int_D\left[(\JJ^T \uu_\kappa)(t+\tau)-(\JJ^T \uu_\kappa)(t)\right]\cdot
\left[(\RR \uu_\kappa)(t+\tau)-(\RR \uu_\kappa)(t)\right]dy \, dt \ \quad \quad \ \ \ \\\nonumber
 \hspace{-5mm} &&\qquad = \int_0^{T-\tau} \hspace{-2mm}\int_D\left[(\sqrt{h}  \uu_\kappa)(t+\tau)-(\sqrt{h} \uu_\kappa)(t)\right]^2dy\, dt+ IV, \quad \mbox{where}\\\nonumber
 \hspace{-5mm} && IV:=\int_0^{T-\tau}\hspace{-2mm} \int_D 2(\sqrt{h} \uu_\kappa)(t+\tau)\cdot (\sqrt{h}  \uu_\kappa)(t)
\\ \hspace{-5mm}&& \qquad\qquad\qquad-(\JJ^T\uu_\kappa)(t)\cdot(\RR\uu_\kappa)(t+\tau)
-(\JJ^T\uu_\kappa)(t+\tau)\cdot(\RR\uu_\kappa)(t)dy\, dt.\nonumber
 \end{eqnarray}
 Due to {\red
$\RR\uu_\kappa=(\sqrt{h}\JJ^{-1})(\sqrt{h}\uu_\kappa)$
   ,} the term $IV$ can be rewritten as
 \begin{eqnarray*}
   IV\hspace{-3mm}&=&\hspace{-3mm}\int_0^{T-\tau} \hspace{-2mm} \int_D(\JJ^T\uu_\kappa)(t+\tau)\cdot
   \left[(\sqrt{h}\JJ^{-1})(t+\tau)-(\sqrt{h}\JJ^{-1})(t)\right](\sqrt{h}\uu_\kappa)(t)
   \\&&\hphantom{\int_0^{T-\tau} \hspace{-2mm}}
  +(\JJ^T\uu_\kappa)(t)\cdot
   \left[(\sqrt{h}\JJ^{-1})(t)-(\sqrt{h}\JJ^{-1})(t+\tau)\right](\sqrt{h}\uu_\kappa)(t+\tau)dy\, dt\\
   \hspace{-3mm}&=&\hspace{-3mm}\int_0^{T-\tau} \hspace{-3mm} \int_D\hspace{-2mm}(\JJ^T\uu_\kappa)(t+\tau)\cdot\partial_t\MM
   (\sqrt{h}\uu_\kappa)(t)
   -(\JJ^T\uu_\kappa)(t)\cdot\partial_t \MM
   (\sqrt{h}\uu_\kappa)(t+\tau)dy\,dt,\\
   && \mbox{where}\ \MM:=\int_t^{t+\tau}\hspace{-3mm}\sqrt{h}(s)\JJ^{-1}(s)ds.
 \end{eqnarray*}
 Finally, taking into account (\ref{timeder_1}) and  (\ref{timeder_2})  we obtain from  (\ref{WF3}),
 \begin{eqnarray}\label{WF4}
\int_0^{T-\tau} \hspace{-3mm} \int_D\left[(\sqrt{h}  \uu_\kappa)(t+\tau)-(\sqrt{h} \uu_\kappa)(t)\right]^2dy+\int_0^L [\sigma_\kappa(t+\tau)-\sigma_\kappa(t)]^2  dy_1\, dt\, \nonumber \\ = I+II+III+IV+\ldots\ \qquad\qquad 
 \end{eqnarray}
 In previous lines we have already shown that   $I+II\leq c\tau$ independently on $\veps$, cf. (\ref{est1}). 
 Now we estimate the additional  terms $III,\, IV$ arising from the (weak) time derivative of $(h\uu_\kappa)$, cf.  (\ref{WF3}),  (\ref{timeder_2}) with $c\tau$ (independently on $\veps$) as well.
  Indeed, due to the assumptions on $h$, (\ref{assumptions}),
 it is obvious that $|\JJ^T|\leq C_{K,\alpha}$, moreover
 $$|\partial_t \MM|\leq \int_t^{t+\tau}|\partial_t(\sqrt{h}\JJ^{-1}(s))|ds \leq  C_{K,\alpha}\int_t^{t+\tau}|h_t(s)|+|h_{ty_1}(s)|ds,$$
 cf. (\ref{matrices}). Therefore
\begin{eqnarray}\label{est2}
 IV&\hspace{-2mm}\leq\hspace{-2mm}& C_{K,\alpha}\int_0^{T-\tau}\hspace{-2mm}\int_t^{t+\tau}\hspace{-4mm}\|h_t(s)\|_{W^{1,\infty}(0,L)}ds \|\uu_\kappa(t)\|_{L^2(D)}\|\uu_\kappa(t+\tau)\|_{L^2(D)}dt\qquad
\qquad  \\\nonumber&\hspace{-2mm}\leq\hspace{-2mm}& C_{K,\alpha} \|\uu_\kappa\|_{L^\infty(0,T;L^2( D))}^2 \tau \int_0^{T-\tau} [\|h_t\|_{W^{1,\infty}(0,L)}]_\tau
 \\\nonumber&\hspace{-2mm}\leq\hspace{-2mm}& \tau C_{K,\alpha} \|\uu_\kappa\|_{L^\infty(0,T;L^2( D))}^2 \|h_t\|_{L^2(0,T;W^{1,\infty}(0,L))} \leq c\tau. 
\end{eqnarray}
Here we have used the fact that $\|[\varphi]_\tau\|_{L^2(0,T-\tau)}\leq \|\varphi\|_{L^2(0,T)}, $
$[\varphi]_\tau(t):=\frac1{\tau}\int_t^{t+\tau}\varphi(s)ds$,
which can be proven by the H\"older inequality and the Fubini theorem as follows.
\begin{eqnarray}\label{Steklov}
  \|[\varphi]_\tau\|^2_{L^2(0,T-\tau)}
  \leq \frac1{\tau} \int_0^{T-\tau}\hspace{-1mm}\int_t^{t+\tau}\varphi^2(s)ds\, dt
  =\frac1{\tau}\int_0^{T-\tau}\hspace{-2mm} \int_0^{T-\tau}\hspace{-5mm}\jedna_{(t,t+\tau)}(s)\varphi^2(s)ds\, dt\nonumber
 \\ =\frac1{\tau}\int_0^{T-\tau}\hspace{-1mm} \varphi^2(s)\int_0^{T-\tau}\hspace{-1mm}\jedna_{(s-\tau,s)}(t)dt\, ds\leq \|\varphi\|^2_{L^2(0,T)} .\hspace{2cm}
  \end{eqnarray}

Analogously we can estimate  term $III$ using {\red $|\RR^{-T}|\leq C_{K,\alpha}$ and \\$|\partial_t \RR^{-T}|\leq  C_{K,\alpha}\|h_t(s)\|_{W^{1,\infty}(0,L)}$} as follows.
\begin{eqnarray}\label{est3}
  III&\hspace{-2mm}\leq\hspace{-2mm}&\int_0^{T-\tau}\hspace{-2mm}\int_D \int_t^{t+\tau}\hspace{-2mm}|(\partial_t\RR^{-T} h\uu_\kappa)(s)|\cdot\Big(|\RR\uu_\kappa(t+\tau)|+|\RR\uu_\kappa(t)|\Big) ds\, dy\, dt\ 
  \\\nonumber&\hspace{-2mm}\leq\hspace{-2mm}& C_{K,\alpha}\hspace{-1mm}\int_0^{T-\tau}\hspace{-3mm}\int_t^{t+\tau}\hspace{-4mm}\|h_t(s)\|_{W^{1,\infty}(0,L)} \|\uu_\kappa(s)\|_{L^2(D)}ds\||\uu_\kappa(t+\tau)|+|\uu(t)|\|_{L^2(D)} dt
  \\\nonumber&\hspace{-2mm}\leq\hspace{-2mm}& \tau  C_{K,\alpha} \|\uu_\kappa\|_{L^\infty(0,T;L^2( D))}^2\int_0^{T-\tau}[\|h_t\|_{W^{1,\infty}(0,L)}]_\tau dt \\\nonumber&\hspace{-2mm}\leq\hspace{-2mm}& \tau  C_{K,\alpha} \|\uu_\kappa\|_{L^\infty(0,T;L^2( D))}^2\|h_t\|_{L^2(0,T;W^{1,\infty}(0,L))} \leq c\tau.
\end{eqnarray}

\medskip

\noindent
Estimates  of remaining terms on the right-hand side of (\ref{WF4}) can be done analogously as in \cite[Section 5.1]{HLN14}, cf. also  \cite[Section 8]{FZ10}, however some further estimates has to be done due to new test functions containing $\RR$, cf.  (\ref{testfunctions1})$_1$. In the  viscous term (\ref{eij}),  
taking into account (\ref{assumptions}),   we additionally obtain terms of type
\begin{equation}\label{visc_est}
  \frac{\mu}{\rho}C_{K,\alpha}\int_0^{T-\tau} \int_t^{t+\tau} \int_D |\nabla \uu_\kappa(y,s)||\partial_{y_1}^2 h(y_1,\theta_1)||\uu_\kappa(y,\theta_2)|\, dy\, ds\, dt,\end{equation}
where $\theta_1:=t,\, t+\tau,$ or $s$, and $\theta_2:=t,$ or $t+\tau$. Rewriting the integral $\int_D dy=\int_0^L\int_0^1 dy_2 \, dy_1$, applying the H\"older inequality and Lemma \ref{lemma:essup}\footnote{valid  for non-solenoidal functions $\uu_\kappa\in \VV$ as well} we can bound the above  terms by
\begin{equation}
  \leq \frac{\mu}{\rho}C_{K,\alpha}\int_0^{T-\tau} \int_t^{t+\tau} \|\nabla \uu_\kappa(\theta_2)\|_{L^2(D)}\| h(\theta_1)\|_{H^2(0,L)}\|\nabla \uu_\kappa(s)\|_{L^2(D)}.\nonumber
  \end{equation}

Applying the H\"older inequality for time integrals and the property  (\ref{Steklov}) we  finally  get for (\ref{visc_est})
\begin{equation}
  \leq \tau \frac{\mu}{\rho}C_{K,\alpha}\| h(\theta_1)\|_{L^\infty(0,T;H^2(0,L))}\|\uu_\kappa\|^2_{L^2(0,T;W^{1,2}(D))}\leq c\tau.\nonumber
  \end{equation}

The convective term (\ref{4.3}) can be estimated using the property (\ref{becko-rozdiel}) by $c\tau$ applying  similar technique using Lemma \ref{lemma:essup}. Let us also recall, that the boundary term containing $\kappa$ is bounded by $c\tau$ as well, where $c$ is some constant independent on $\kappa$, for more details see \cite[Section 5.1]{HLN14}, \cite[Section 8]{FZ10}.

Finally,  due to estimates (\ref{est1}), (\ref{est2}), (\ref{est3}) and above described estimates of remaining terms
on the right-hand side of (\ref{WF4}), 
we can conclude that
\begin{equation}\label{est_final}
  \int_0^{T-\tau} \hspace{-2mm}\int_D\left[(\sqrt{h}  \uu_\kappa)(t+\tau)-(\sqrt{h} \uu_\kappa)(t)\right]^2 dy+\int_0^L [\sigma_\kappa(t+\tau)-\sigma_\kappa(t)]^2 \, dy_1\, dt \leq c\tau
  \end{equation}
with some constant $c$ dependent on $\alpha, K, T$, but  independent on $\veps$ and $\kappa$. 
The integral equicontinuity estimate (\ref{est_final})  together with the compactness argument \cite[Lemma 1.9]{AL} provide us the strong convergences $\uu_\kappa\to \uu$ in $L^r((0,T)\times D), 1\leq r<4$ and $\sigma_\kappa=\partial_t\eta_\kappa \to \sigma=\partial_t\eta$ in $L^2((0,T)\times (0,L)), 1\leq s<6$, see \cite[p. 223]{HLN14} or \cite[p. 36]{FZ10} for more details.  Finally, the limiting process  for  $\kappa \to \infty$ in (\ref{WF1}) with the test function $\xi(y_1,t)=E\pp_2(y_1,1,t)\in  H^2_0(0,L)$
 completes the proof of Theorem 5.1 \cite{HLN14}. 
\qquad $ \square$

\bigskip 



\bibliographystyle{elsarticle-num}
\bibliography{all}





\end{document}